\documentclass[10pt,a4paper,reqno]{amsproc}
\usepackage{longtable,cite}
\usepackage[a4paper, top=1in, bottom=1in, left=0.8in, right=0.8in]{geometry}
\usepackage{lscape}
\usepackage{multirow}
\usepackage[pdftex,bookmarks=true]{hyperref}
\newtheorem{theorem}{Theorem}[section]
\newtheorem{lemma}[theorem]{Lemma}

\newtheorem{proposition}[theorem]{Proposition}
\newtheorem{corollary}[theorem]{Corollary}
\theoremstyle{definition}
\newtheorem{definition}[theorem]{Definition}

\theoremstyle{remark}
\newtheorem{remark}[theorem]{Remark}
\numberwithin{equation}{section}
\usepackage{hyperref}
\usepackage{tikz}
\usetikzlibrary{shapes, arrows}
\usetikzlibrary{chains,arrows.meta}
\begin{document}
	\title [ Inequality, UPs and their structural analysis for OLCT and  its quaternion extension
	]{Inequality, uncertainty principles and their structural analysis for offset linear canonical transform and  its quaternion extension} 
	\author{Gita Rani Mahato}
	\author{Sarga Varghese}
	\author{Manab Kundu}
	
	\address{Department of Mathematics, SRM University AP, Amaravati-522240, India}
	
	\email{\hfill \break
		gitamahato1158$@$gmail.com (Gita Rani Mahato),
		\hfill \break sargavarghese$@$gmail.com (Sarga Varghese),
		\hfill \break manabiitism17$@$gmail.com (Manab Kundu-Corresponding author))}
	
	\thanks{Corresponding author: Manab Kundu}
	
	\date{}
	
	\begin{abstract}
		This work undertakes a twofold investigation. In the first part, we examine the inequalities and uncertainty principles in the framework of offset linear canonical transform (OLCT), with particular attention to its scaling and shifting effects. Theoretical developments are complemented by numerical simulations that substantiate and illustrate the analytical results. In the second part, we establish the connection of quaternion offset linear canonical transform (QOLCT) and the OLCT by employing the orthogonal plane split (OPS) approach. Through this approach, the inequalities and uncertainty principles derived for the OLCT are extended to the QOLCT. Moreover, the computational methods designed for the OLCT may systematically adapted to facilitate the numerical implementation of the QOLCT using this connection between OLCT and QOLCT.
	\end{abstract}
	\maketitle
	\section{Introduction}
	The quaternion Fourier transform (QFT) \cite{ft} extends the classical Fourier transform \cite{pt} to quaternion-valued functions, enabling the analysis of multi-dimensional signals such as color images, vector fields, and polarized waveforms \cite{ellbook}. By exploiting the algebraic structure of quaternions, the QFT provides a natural framework for handling signals with multiple components \cite{gil}. However, like its complex-valued counterpart, the QFT has limitations in representing time-varying or non-stationary signals \cite{ft2}. To address these shortcomings, the quaternion linear canonical transform (QLCT) was proposed as a generalization of the QFT \cite{lct, qlct2}. The QLCT introduces additional transformation parameters that allow for operations such as scaling, chirp modulation, and fractional-domain transformations, thereby enhancing time–frequency analysis in the quaternionic setting \cite{kou}. Despite these improvements, the QLCT does not explicitly account for time and frequency shifts \cite{hitzer2023}. To overcome this limitation, Abe and Sheridan (1994) introduced the offset linear canonical transform (OLCT) \cite{Abe, xqq, ofrst}, which incorporates both linear canonical parameters and time–frequency offset terms. The OLCT provides a more comprehensive representation of signals undergoing complex transformations \cite{stern} and has proven to be a powerful tool in both theoretical studies and practical applications, particularly in optics and signal processing \cite{onural, huo}.
	\\
	\\
	More recently, the OLCT framework has been extended to higher-dimensional and non-commutative domains through the quaternionic offset linear canonical transform (QOLCT) \cite{pln}. The QOLCT employs quaternion algebra a number system that generalizes complex numbers to represent multi-component signals such as color images and vector-valued data within a unified framework \cite{hitzer2009}. By embedding additional offset parameters into the QLCT structure \cite{hd}, the QOLCT enables simultaneous shifting and modulation, thereby providing greater flexibility for practical applications. Importantly, both the QFT and QLCT emerge as special cases of the QOLCT under suitable parameter choices, establishing it as a comprehensive and versatile tool for quaternionic signal analysis \cite{zz,huk}. This advancement significantly strengthens the processing of multi-channel and high-dimensional data, while preserving fundamental mathematical properties such as invertibility, unitarity, and energy conservation \cite{blackledge}. With the introduction of such generalized transforms, it becomes essential to investigate the theoretical constraints they impose on signal representation \cite{wilson}. Among these, the uncertainty principle remains one of the most critical results, as it establishes fundamental lower bounds on the simultaneous localization of a signal and its transform. These principles are indispensable for understanding signal behavior, particularly in high-dimensional and multi-channel contexts \cite{uryn, kundu2025}.
	\\
	\\
	Uncertainty principles have been actively extended into higher-dimensional and non‑commutative domains to address the shortcomings of classical transforms. Notably, the QOLCT has become a key framework for analyzing color images, vector fields, and other multi-component signals. Several recent works have advanced this field. El Haoui et al. \cite{em} introduced the QOLCT, established its connection with the QFT, and generalized key uncertainty principles such as Heisenberg Weyl, Hardy, Beurling, and logarithmic types within the QOLCT domain. Bahri et al. \cite{bm} defined QOLCT through its connection to the QLCT and presented the uncertainty principles for the quaternion linear canonical transform (QLCT)by generalizing existing results \cite{lian, mk}. 
	\\
	\\
	Scaling and shifting are pivotal properties of integral transforms, as they allow control over signal resolution, adjustment of frequency content, and modeling of transformations in optical systems \cite{shift}. These operations are therefore fundamental in time frequency localization, signal processing, image analysis, and in establishing the mathematical foundations of uncertainty principles. More recently, Chen et al. established 
	$L_p$-type Heisenberg–Pauli–Weyl uncertainty principles for the fractional Fourier transform (FrFT) and analyzed their behavior under scaling and shifting \cite{Chen2025}. Building on this work, Mai et al. generalized these uncertainty principles to the linear canonical transform (LCT) and further extended them to random signals, where scaling and shifting properties enhanced their applications in communications and optics \cite{Mai2025}. To the best of our knowledge, the effects of scaling and shifting in the OLCT, a generalized form of the FrFT and LCT, have not been investigated. Moreover, these effects remain unexplored in the quaternion domain, primarily due to the non-commutative nature of quaternions and the computational complexity of the kernel. Motivated by this gap, we aim to develop a simple yet effective framework for computation and simulation in quaternion space. In this paper, our focus is primarily on the following aspects:
	\\
	\begin{itemize}
		\item Derivation of several important inequalities and uncertainty principles for the 2D OLCT in $\mathbb{R}^2$.
		\item Investigation of how the shifting and scaling properties influence the behavior of two-dimensional inequalities and uncertainty principles.
		\item Validation of the derived results through numerical simulations supported by graphical illustrations.
		\item Establishing the connection between the 2D OLCT and the QOLCT using the OPS method.
		\item Derivation of inequalities such as Pitt’s inequality, the Hausdorff–Young inequality, and various uncertainty principles with the aid of 2D OLCT and the OPS method.
	\end{itemize}
	These contributions will not only enable us to investigate the effects of scaling and shifting in the quaternion domain but also provide computational tools for further simulations in the generalized  quaternion integral transform framework.
	\\
	\\
	The paper is arranged as follows. Section 2 covers the core definitions and preliminaries related to 2D OLCT. In Section 3, inequalities and uncertainty principles for 2D OLCT, are explored. Section 4 is dedicated to analyzing the effect of scaling and shifting operations on uncertainty principles and inequalities. Section 5 discuss numerical simulation with graphical representation. Section 6, introduces quaternions and relations between 2D OLCT and QOLCT.In Section 7, we present the formulation and proofs of multiple inequalities and uncertainty principles arising in the QOLCT domain. Section 8 provides the concluding remarks.
	
	\section{Preliminaries}
	In this section, we revisit some fundamental definitions, relations and properties of 2D OLCT which help us  understanding its analytical structure.
	
	\begin{definition}
		Let the function $f \in L^1(\mathbb{R})$. Then Fourier transform of $f$  is  defined as \cite{ap}
		\begin{eqnarray}
			({F}f)(u) = \frac{1}{\sqrt{2\pi}} \int_{\mathbb{R}} e^{-ixu}f(x)dx,\hspace{3mm} \forall u\in \mathbb{R}.
		\end{eqnarray}
		The inverse Fourier transform can be written as
		\begin{eqnarray}
			f(x)=\frac{1}{\sqrt{2\pi}} \int_{\mathbb{R}} e^{ixu}({F}f)(u)du,\hspace{3mm} \forall x\in \mathbb{R}.
		\end{eqnarray}
	\end{definition}
	\begin{definition} Let $f\in   L^1(\mathbb{R}) $ then OLCT of $f$ can be defined as   \cite{gm} 
		\begin{eqnarray}\label{2.3}
			\small
			O_M[f(t)](u)=
			\begin{cases}
				\int_{\mathbb{R}}f(t)h_M(t,u)dt & \hspace{-4mm}, ~~ b\neq0\\
				\sqrt{d}~e^{i\frac{cd}{2}(u-\tau)^2+i\eta}f[d(u-\tau)] & \hspace{-0.3cm}, ~~b=0, 
			\end{cases}
		\end{eqnarray}
		where
		\begin{eqnarray}
			h_M (t ,u)= \sqrt{\frac{1}{2\pi i b}}~e^{\frac{i}{2b}(a t^2 + 2 t (\tau-u)-2u(d\tau-b\eta)+ du^2+ d\tau^2)},
		\end{eqnarray}
		is the kernel of OLCT, with	$ a, b, c, d, \tau, \eta \in \mathbb{R}$ and $ad - bc = 1$. Here, we will restrict our focus to the OLCT case where 
		$b\neq0$ .
		Then its inverse transformation can be defined as 
		\begin{eqnarray}
			f(x)=  \Big(\mathcal{O}_{M^{-1}} (\mathcal{O}_{M} f)\Big)(x)= C \int_{\mathbb{R}}  \mathcal{K}_{M^{-1}} (x, u) (\mathcal{O}_{M} f)(u) du,
		\end{eqnarray}
		where 
		\begin{eqnarray*}
			&&{M^{-1}=(d,-b,-c,a,b\eta-d\tau,c\tau-a\eta) }
			~\text{and} \\ && C= e^{\frac{i}{b}(cd\tau^2-2ad\tau\eta+ab\eta^2)}.
		\end{eqnarray*}
		
	\end{definition}
	
	\begin{definition}\label{olct}
		Let  $f \in L^1(\mathbb{R}^2) $. Then 2D OLCT of a function $f$ is defined as 
		\begin{eqnarray}
			\small
			F_A(u)= O_{M_1 M_2}[f(t)](u)=		
			\int_{\mathbb{R}^2} \mathcal{K}_{M_1}(u_1,t_1) \mathcal{K}_{M_2}(u_2,t_2)f(t) dt,
		\end{eqnarray}
		where
		\begin{eqnarray}
			\mathcal{K}_{M_r}(u_r,t_r)	=
			\begin{cases}
				\sqrt{\frac{1}{2\pi i b_r}}~e^{\frac{i}{2b_r}(a_r t_r^2 + 2 t_r (\tau_r-u_r)-2u_r(d_r\tau_r-b_r\eta_r)+ d_ru_r^2+ d_r\tau_r^2)}, & \hspace{-4mm} ~~ b_r\neq0\\
				\sqrt{d_r}~e^{i\frac{c_rd_r}{2}(u_r-\tau_r)^2+i\eta_r u}f[d_r(u_r-\tau_r)] & \hspace{-0.3cm}, ~~b_r=0.
			\end{cases}
		\end{eqnarray}
	\end{definition}

	\begin{lemma}\label{shifting}
		Shifting property for 2D OLCT: If $f\in L^1(\mathbb{R}^2)$  and $\mathcal{O}_{M_1M_2}f(u)$ is 2D OLCT of $f$ and $\alpha\in \mathbb{R}^2$ then
		\begin{eqnarray*}
			\mathcal{O}_{M_1M_2}[f(t-\alpha)](u)&=&\mathcal{O}_{M_1M_2}[f(t)](u-a\alpha)C_1
		\end{eqnarray*}
		where
		\begin{eqnarray*}
			C_1=e^{\frac{i}{2b_1}[2\alpha_1(\tau_1-u_1)-a_1\alpha_1(d_1\tau_1-b_1\eta_1)]}e^{\frac{i}{2b_2}[2\alpha_2(\tau_2-u_2)-a_2\alpha_2(d_2\tau_2-b_2\eta_2)]}
		\end{eqnarray*}
	\end{lemma}
	\begin{proof} Let us consider a shifted function $f( t-\alpha t)$. Then
		\begin{eqnarray*}
			\mathcal{O}_{M_1M_2}[f(t-\alpha)](u)&=&		
			\int_{\mathbb{R^2}} \mathcal{K}_{M_1}(u_1,t_1) \mathcal{K}_{M_2}(u_2,t_2)f(t-\alpha) dt\\
			&=&{\frac{1}{2\pi(-b_1b_2)}}\int_{\mathbb{R}^2}e^{\frac{i}{2b_1}(a_1 t_1^2 + 2 t_1 (\tau_1-u_1)-2u_1(d_1\tau_1-b_1\eta_1)+ d_1u_1^2+ d_1\tau_1^2)}\\&\times&e^{\frac{i}{2b_2}(a_2 t_2^2 + 2 t_2 (\tau_2-u_2)-2u_2(d_2\tau_2-b_2\eta_2)+ d_2u_2^2+ d_2\tau_2^2)}f(t-\alpha)dt.
		\end{eqnarray*}
		Changing the variable $t-\alpha=x$, we get
		\begin{eqnarray*}
			\mathcal{O}_{M_1M_2}[f(t-\alpha)](u)&=&{\frac{1}{2\pi(-b_1b_2)}}\int_{\mathbb{R}^2}e^{\frac{i}{2b_1}(a_1 (x_1+\alpha_1)^2 + 2 (x_1+\alpha_1) (\tau_1-u_1)-2u_1(d_1\tau_1-b_1\eta_1)+ d_1u_1^2+ d_1\tau_1^2)}\\&\times&e^{\frac{i}{2b_2}(a_2 (x_2+\alpha_2)^2 + 2 (x+\alpha_2) (\tau_2-u_2)-2u_2(d_2\tau_2-b_2\eta_2)+ d_2u_2^2+ d_2\tau_2^2)}f(x)dx\\
			&=&{\frac{1}{2\pi(-b_1b_2)}}\int_{\mathbb{R}^2}e^{\frac{i}{2b_1}(a_1 (x_1+\alpha_1)^2 + 2 (x_1+\alpha_1) (\tau_1-u_1)-2u_1(d_1\tau_1-b_1\eta_1)+ d_1u_1^2+ d_1\tau_1^2)}\\&\times&e^{\frac{i}{2b_2}(a_2 (x_2+\alpha_2)^2 + 2 (x+\alpha_2) (\tau_2-u_2)-2u_2(d_2\tau_2-b_2\eta_2)+ d_2u_2^2+ d_2\tau_2^2)}f(x)dx\\
			&=&\mathcal{O}_{M_1M_2}f(t)(u-a\alpha)e^{\frac{i}{2b_1}[2\alpha_1(\tau_1-u_1)-a_1\alpha_1(d_1\tau_1-b_1\eta_1)]}\\&\times&e^{\frac{i}{2b_2}[2\alpha_2(\tau_2-u_2)-a_2\alpha_2(d_2\tau_2-b_2\eta_2)]}\\&=&\mathcal{O}_{M_1M_2}[f(t)](u-a\alpha)C_1.
		\end{eqnarray*}
	\end{proof}
	\begin{lemma}\label{scaling}
		Scaling property for 2D OLCT: If $f\in L^1(\mathbb{R}^2)$  and $\mathcal{O}_{M_1M_2}f(u)$ is 2D OLCT of $f$, then 
		\begin{eqnarray*}
			\mathcal{O}_{M_1M_2}[f(\alpha t)]=\frac{1}{\alpha}\mathcal{O}_{M_1'M_2'}(f)(\frac{u}{\alpha})
		\end{eqnarray*} 
		where $M'_1=(\frac{a_1}{\alpha^2},b_1,c_1,d_1\alpha^2,\frac{\tau_1}{\alpha},\alpha \eta_1)$ and  $ M'_2=(\frac{a_2}{\alpha^2},b_2,c_2,d_2\alpha^2,\frac{\tau_2}{\alpha},\alpha \eta_2).$
	\end{lemma}
	\begin{proof} Let us consider the 2D OLCT of scaled function $f(\alpha t)$:
		\begin{eqnarray*}
			\mathcal{O}_{M_1M_2}[f(\alpha t)](u) &=&		
			\int_{\mathbb{R^2}} \mathcal{K}_{M_1}(u_1,t_1) \mathcal{K}_{M_2}(u_2,t_2)f(\alpha t) dt\\
			&=&{\frac{1}{2\pi(-b_1b_2)}}\int_{\mathbb{R}^2}e^{\frac{i}{2b_1}(a_1 t_1^2 + 2 t_1 (\tau_1-u_1)-2u_1(d_1\tau_1-b_1\eta_1)+ d_1u_1^2+ d_1\tau_1^2)}\\&\times&e^{\frac{i}{2b_2}(a_2 t_2^2 + 2 t_2 (\tau_2-u_2)-2u_2(d_2\tau_2-b_2\eta_2)+ d_2u_2^2+ d_2\tau_2^2)}f(\alpha t)dt.
		\end{eqnarray*}
		Changing the variable $\alpha t =x$, we get
		\begin{eqnarray*}
			\mathcal{O}_{M_1M_2}[f(\alpha t)](u)	&=&{\frac{1}{2\pi(-b_1b_2)}}\int_{\mathbb{R}^2}e^{\frac{i}{2b_1}(a_1 (\frac{x_1}{\alpha_1})^2 + 2 (\frac{x_1}{\alpha_1}) (\tau_1-u_1)-2u_1(d_1\tau_1-b_1\eta_1)+ d_1u_1^2+ d_1\tau_1^2)}\\&\times&e^{\frac{i}{2b_2}(a_2 (\frac{x_2}{\alpha_2})^2 + 2 (\frac{x_2}{\alpha_2}) (\tau_2-u_2)-2u_2(d_2\tau_2-b_2\eta_2)+ d_2u_2^2+ d_2\tau_2^2)}f(x)dx\frac{1}{\alpha}\\
			&=&\frac{1}{\alpha}\mathcal{O}_{M_1'M_2'}(f)(\frac{u}{\alpha}),
		\end{eqnarray*}
		where $M'_1=(\frac{a_1}{\alpha^2},b_1,c_1,d_1\alpha^2,\frac{\tau_1}{\alpha},\alpha \eta_1)$ and  $ M'_2=(\frac{a_2}{\alpha^2},b_2,c_2,d_2\alpha^2,\frac{\tau_2}{\alpha},\alpha \eta_2).$
	\end{proof}
	\begin{lemma}\label{dup}
		Differential property for 2D OLCT : If $f\in L^1(\mathbb{R}^2)$ is infinitely differentiable  and $\mathcal{O}_{M_1M_2}f(u)$ is 2D OLCT of $f$. Then
		\begin{eqnarray*}
			\left(\mathcal{O}_{M_1M_2}D_{t_1^mt_2^n}^{M_1M_2 }f\right)(u_1,u_2)= \left(\frac{-iu_1}{b_1}\right)^m\left(\frac{-iu_2}{b_2}\right)^n\left(\mathcal{O}_{M_1M_2}f\right)(u_1,u_2),
		\end{eqnarray*}
		where \begin{eqnarray*}
			D_{t_1^mt_2^n}^{M_1M_2 }=\Delta_{t_1^m}^{iM_1}\Delta_{t_2^n}^{iM_2}\\ \Delta_{t_1^m}^{iM_1}=(-1)^m\left(\frac{\partial }{\partial t_1}+\frac{i}{b_1}(a_1t_1+\tau_1)\right), \\
			\Delta_{t_2^n}^{iM_2}=(-1)^n\left(\frac{\partial }{\partial t_2}+\frac{i}{b_2}(a_2t_2+\tau_2)\right) .
		\end{eqnarray*}
	\end{lemma}
	\begin{proof}
		The derivative type property of the 2D OLCT is established using the principle of mathematical induction. \\
		For m=n=0, the proof is straightforward and holds directly by definition.  \\
		Now, we consider  for m=1 and n=0
		\begin{eqnarray*}
			\left(\mathcal{O}_{M_1M_2}D_{t_1^1t_2^0}^{M_1M_2 }f\right)(u_1,u_2)&=&\int_{\mathbb{R}^2}\mathcal{K}_{M_1}(u_1,t_1)D_{t_1^1t_2^0}^{M_1M_2 }f(t)\mathcal{K}_{M_2}(u_2,t_2)dt
			\\&=&\int_{\mathbb{R}^2}\mathcal{K}_{M_1}(u_1,t_1)\Delta_{t_1^1}^{iM_1 }f(t)\mathcal{K}_{M_2}(u_2,t_2)dt\\&=&(-\frac{iu_1}{b_1})\int_{\mathbb{R}^2}\mathcal{K}_{M_1}(u_1,t_1)f(t)\mathcal{K}_{M_2}(u_2,t_2)dt
		\end{eqnarray*}
		Using the definition of 2D OlCT \ref{olct}
		\begin{eqnarray}\label{mn}
			\left(\mathcal{O}_{M_1M_2}D_{t_1^1t_2^0}^{M_1M_2 }f\right)(u_1,u_2)	&=&(-\frac{iu_1}{b_1})\left(\mathcal{O}_{M_1M_2}f\right)(u_1,u_2).
		\end{eqnarray}
		In similar way we can proceed m=0, n=1 we get
		\begin{eqnarray}\label{nm}
			\left(\mathcal{O}_{M_1M_2}D_{t_1^0t_2^1}^{M_1M_2 }f\right)(u_1,u_2)&=&(-\frac{iu_2}{b_2})\left(\mathcal{O}_{M_1M_2}f\right)(u_1,u_2).
		\end{eqnarray}
		Combining \eqref{mn} and \eqref{nm}, we get the required result.\\
		
		The proof for general $m,n\in \mathbb{N}$ follows directly by induction.
	\end{proof}
	
	\subsection{Connection between the 2D OLCT and the 2D Fourier Transform}\ \\
	Here, a relation is established between 2D OLCT and 2D FT.\\
	
	Let  $f \in L^1(\mathbb{R}^2) $. Denote $O_{M_1 M_2}f$ as 2D OLCT of $f$, $\mathcal{F}f$ as the 2D FT. Then
	\begin{eqnarray*}
		&&O_{M_1 M_2}[f(t)](u)\\&=&		
		\int_{\mathbb{R}^2} \mathcal{K}_{M_1}(u_1,t_1) \mathcal{K}_{M_2}(u_2,t_2)f(t) dt\\&=& \frac{1}{(\sqrt{2 \pi})^2 \sqrt{-b_1b_2}}\int_{\mathbb{R}^2}e^{{\sum_{r=1}^{2}}{\frac{i}{2b_r}(a_r t_r^2 + 2 t_r (\tau_r-u_r)-2u_r(d_r\tau_r-b_r\eta_r)+ d_ru_r^2+ d_r\tau_r^2)}}f(t)dt,\\&=&
		\frac{1}{({2 \pi}) \sqrt{-b_1b_2}}\int_{\mathbb{R}^2}e^{{\sum_{r=1}^{2}}{\frac{i}{2b_r}(a_r t_r^2 + 2 t_r (\tau_r-u_r)-2u_r(d_r\tau_r-b_r\eta_r)+ d_ru_r^2+ d_r\tau_r^2)}}\\ &\times& e^{{\sum_{r=1}^{2}}{\frac{i}{2b_r}(-2t_ru_r)}}f(t)dt\\&=&
		\frac{1}{({2 \pi}) \sqrt{-b_1b_2}}\int_{\mathbb{R}^2}e^{{\sum_{r=1}^{2}}{\frac{i}{2b_r}(a_r t_r^2 + 2 t_r \tau_r+d_r\tau_r^2)}}~e^{{\sum_{r=1}^{2}}{\frac{i}{2b_r}}(-2u_r(d_r\tau_r-b_r\eta_r)+ d_ru_r^2)}\\ &\times& e^{{\sum_{r=1}^{2}}{{it_r}(\frac{u_r}{b_r})}}f(t)dt\\&=&
		\frac{1}{ \sqrt{-b_1b_2}}~e^{{\sum_{r=1}^{2}}{\frac{i}{2b_r}(-2u_r(d_r\tau_r-b_r\eta_r)+ d_ru_r^2+ d_r\tau_r^2)}}\mathcal{F}\{e^{{\sum_{r=1}^{2}}{\frac{i}{2b_r}(a_r t_r^2 + 2 t_r \tau_r)}}f(t)\}(\frac{u}{b}).
	\end{eqnarray*}
	\section{Inequalities and Uncertainty principle}
	In this section, we study the uncertainty principle and important  inequalities related to 2D OLCT. These results help us to understand the limitations of how well a signal can be concentrated in both the original and transform domains. Furthermore using these results, we establish the UPs and inequalities for QOLCT in section 7.
	
	\subsection{Sharp-Young Hausdorff inequality for 2D OLCT}
	\begin{theorem}\label{sharp}
		Let $f \in L^p(\mathbb{R}^2)$ and $q$ be such that
		$\frac{1}{p} + \frac{1}{q} = 1$. Then
		\begin{eqnarray}
			|| \mathcal{O}_{M_1,M_2} f ||_{L^q(\mathbb{R}^2)} \leq \mathcal{K} \|f\|_{L^p(\mathbb{R}^2)},
		\end{eqnarray}
		where \begin{eqnarray*}
			\mathcal{K} = |b_1 b_2|^{\frac{1}{q} - \frac{1}{2}} \left( \frac{p^{1/p}}{q^{1/q}} \right) (2\pi)^{\frac{1}{q} - \frac{1}{p}}.
		\end{eqnarray*}
	\end{theorem}
	\begin{proof}
		
		The sharp Young–Hausdorff inequality for a function and its Fourier transform in $\mathbb{R}^2$ is stated as follows \cite{zz}.
		\begin{eqnarray}
			\left( \int_{\mathbb{R}^2} |(\mathcal{F}f)(u)|^q du \right)^{1/q} &\leq &\left( \frac{p^{1/p}}{q^{1/q}} \right) (2\pi)^{\frac{1}{q} - \frac{1}{p}} \left( \int_{\mathbb{R}^2} |f(t)|^p dt \right)^{1/p}.
		\end{eqnarray}
		Replacing $u = (u_1, u_2)$ by $\left( \frac{u_1}{b_1}, \frac{u_2}{b_2} \right)$, the inequality becomes
		\begin{eqnarray*}
			\left( \int_{\mathbb{R}^2} |(\mathcal{F}f)\left(\frac{u}{b}\right)|^q \left| \frac{1}{b_1 b_2} \right| du \right)^{1/q} &\leq& \left( \frac{p^{1/p}}{q^{1/q}} \right) (2\pi)^{\frac{1}{q} - \frac{1}{p}} \left( \int_{\mathbb{R}^2} |f(t)|^p dt \right)^{1/p}.
		\end{eqnarray*} 
		Rearranging the terms, we get  
		\begin{eqnarray*}
			\left( \int_{\mathbb{R}^2} |(\mathcal{F}f)\left(\frac{u}{b}\right)|^q du \right)^{1/q} &\leq& |b_1 b_2|^{\frac{1}{q}} \left( \frac{p^{1/p}}{q^{1/q}} \right) (2\pi)^{\frac{1}{q} - \frac{1}{p}} \left( \int_{\mathbb{R}^2} |f(t)|^p dt \right)^{1/p}.		
		\end{eqnarray*}
		Substituting	$e^{\sum_{k=1}^{2}{\frac{j}{2b_k} \left( a_k t_k^2 + 2 t_k \tau_k \right)}} f(t)$ in place of $f(t)$, we get
		\begin{eqnarray*}
			&&	\left( \int_{\mathbb{R}^2} \left| \mathcal{F} \left[  e^{\sum_{k=1}^2\frac{j}{2b_k} \left( a_k t_k^2 + 2 t_k \tau_k \right)} f(t) \right] \right|^q du \right)^{1/q} \\&\leq&|b_1 b_2|^{1/q} \left( \frac{p^{1/p}}{q^{1/q}} \right) (2\pi)^{\frac{1}{q} - \frac{1}{p}} \left( \int_{\mathbb{R}^2} \left|  e^{\sum_{k=1}^2\frac{j}{2b_k} \left( a_k t_k^2 + 2 t_k \tau_k \right)} f(t) \right|^p dt \right)^{1/p}.
		\end{eqnarray*}
		Using the connection between the 2D OLCT and the 2D FT, we obtain
		\begin{eqnarray*}
			&&\left(\left|{\sqrt{-b_1b_2}}~\int_{\mathbb{R}^2}e^{-{\sum_{k=1}^{2}}{\frac{i}{2b_k}(-2u_k(d_k\tau_k-b_k\eta_k)+ d_ku_k^2+ d_k\tau_k^2)}}(\mathcal{O}_{M_1M_2} f)(u)\right|^{q}du\right)^{\frac{1}{q}}
			\\&\leq&|b_1 b_2|^{1/q} \left( \frac{p^{1/p}}{q^{1/q}} \right) (2\pi)^{\frac{1}{q} - \frac{1}{p}} \left( \int_{\mathbb{R}^2} \left|  e^{\sum_{k=1}^2\frac{j}{2b_k} \left( a_k t_k^2 + 2 t_k \tau_k \right)} f(t) \right|^p dt \right)^{1/p}.
		\end{eqnarray*}
		Rearranging the terms, we get
		\begin{eqnarray*}
			&&\left(|{\sqrt{-b_1b_2}}\int_{\mathbb{R}^2}\mathcal{O}_{M_1M_2} (f)(u)|^{q}du\right)^{\frac{1}{q}}\leq|b_1 b_2|^{\frac{1}{q}} \left( \frac{p^{1/p}}{q^{1/q}} \right) (2\pi)^{\frac{1}{q} - \frac{1}{p}} \left( \int_{\mathbb{R}^2} \left| f(t) \right|^p dt \right)^{1/p}.
		\end{eqnarray*}
		Simplifying the inequality, we obtain
		\begin{eqnarray*}
			&&\left(\int_{\mathbb{R}^2}|\mathcal{O}_{M_1M_2} (f)(u)|^{q}du\right)^{\frac{1}{q}}\leq|b_1 b_2|^{\frac{1}{q}-\frac{1}{2}} \left( \frac{p^{1/p}}{q^{1/q}} \right) (2\pi)^{\frac{1}{q} - \frac{1}{p}} \left( \int_{\mathbb{R}^2} \left| f(t) \right|^p dt \right)^{1/p}.
		\end{eqnarray*}
		That is 
		\begin{eqnarray*}
			||\mathcal{O}_{M_1M_2}f||_{L^{q}(\mathbb{R}^2)}\leq|b_1 b_2|^{\frac{1}{q}-\frac{1}{2}} \left( \frac{p^{1/p}}{q^{1/q}} \right) (2\pi)^{\frac{1}{q} - \frac{1}{p}} ||f||_{L^p(\mathbb{R}^2)}. 
		\end{eqnarray*}
	\end{proof}
	This completes the proof.
	\subsection{Pitt's inequality for 2D OLCT}
	\begin{proposition}
		\cite{pt} For $f\in \mathbb{S}(\mathbb{R}^2)$ and $0\leq \lambda<2$ the followings hold
		\begin{eqnarray}\label{3.44}
			\int_{\mathbb{R}^2}|u|^{-\lambda}|\mathcal{O}_{M_1M_2}f(u)|^2du\leq \frac{C_\lambda}{|b_1b_2|^{-\lambda}}\int_{\mathbb{R}^2}|t|^\lambda|f(t)|^2dt,
		\end{eqnarray} 
		where $ C_\lambda =\frac{ \Gamma(\frac{1-\lambda}{4})}{\Gamma(\frac{1+\lambda}{4})}$ and  $\mathbb{S}(\mathbb{R}^2)$ is Schwartz space.
	\end{proposition}
	\begin{proof}
		Pitt's inequality  for 2D FT is given by
		\begin{eqnarray*}\label{pit}
			\int_{\mathbb{R}^2}|u|^{-\lambda}|\mathcal{F}f(u)|^2du\leq C_\lambda \int_{\mathbb{R}^2}|t|^\lambda|f(t)|^2dt,
		\end{eqnarray*}
		where $ C_\lambda =\frac{ \Gamma(\frac{1-\lambda}{4})}{\Gamma(\frac{1+\lambda}{4})}$. \\ Replacing $u$ by $\frac{u}{b}$, we get
		\begin{eqnarray*}
			\frac{1}{|b_1b_2|}\int_{\mathbb{R}^2}|\frac{u}{b}\big|^{-\lambda}\big|\mathcal{F}f(\frac{u}{b})|^2du\leq C_\lambda\int_{\mathbb{R}^2}|t|^\lambda|f(t)|^2dt.
		\end{eqnarray*}
		Substituting	$e^{\sum_{k=1}^{2}{\frac{j}{2b_k} \left( a_k t_k^2 + 2 t_k \tau_k \right)}} f(t)$ in place of $f(t)$, we get
		\begin{eqnarray*}
			\frac{1}{|b_1b_2|}\int_{\mathbb{R}^2}\bigg|\frac{u}{b}\bigg|^{-\lambda}\bigg|\mathcal{F}e^{\sum_{k=1}^{2}{\frac{j}{2b_k} \left( a_k t_k^2 + 2 t_k \tau_k \right)}} f(t)(\frac{u}{b})\bigg|^2du\leq C_\lambda\int_{\mathbb{R}^2}|t|^\lambda|e^{\sum_{k=1}^{2}{\frac{j}{2b_k} \left( a_k t_k^2 + 2 t_k \tau_k \right)}} f(t)|^2dt.
		\end{eqnarray*}
		Using  the relation between between 2D FT and 2D OLCT, we have
		\begin{eqnarray*}
			&&\int_{\mathbb{R}^2}\bigg|\frac{u}{b}\bigg|^{-\lambda}\bigg|{\sqrt{-b_1b_2}}~e^{-{\sum_{k=1}^{2}}{\frac{i}{2b_k}(-2u_k(d_k\tau_k-b_k\eta_k)+ d_ku_k^2+ d_k\tau_k^2)}}(\mathcal{O}_{M_1M_2} f)(u)\bigg|^2du\leq \frac{C_\lambda}{|b_1b_2|^{-1}}\int_{\mathbb{R}^2}|t|^\lambda|f(t)|^2dt.
		\end{eqnarray*} 
		Rearranging the inequality, we can write
		\begin{eqnarray*}
			&&\int_{\mathbb{R}^2}|u|^{-\lambda}|\mathcal{O}_{M_1M_2}|^2du \leq  \frac{C_\lambda }{|b_1b_2|^\lambda}\int_{\mathbb{R}^2}|t|^\lambda |f(t)|^2dt.
		\end{eqnarray*}
		This completes the proof.
	\end{proof}
	\begin{remark}
		At $\lambda=0$, Pitt’s inequality becomes an equality, and differentiating it at this point yields the logarithmic uncertainty principle for the 2D OLCT. That is
		\begin{eqnarray}\label{2dlogrithmic}
			\int_{\mathbb{R}^2}\ln|u||\mathcal{O}_{M_1M_2}f(u)|^2du+\int_{\mathbb{R}^2}\ln |t||f(t)|^2dt\geq K'_0\int_{\mathbb{R}^2}|f(t)|^2dt,
		\end{eqnarray} 
		where $K'_0=\frac{d}{d\lambda}K_\lambda$ at $\lambda = 0$ and $K_\lambda = \frac{C_\lambda }{|b_1b_2|^\lambda}$.
	\end{remark}
	\subsection{Entropy uncertainty principle}\label{eup}\ \\
	I.I. Hirschman first introduced the idea of entropic uncertainty in 1957 \cite{hir}, who proposed using Shannon entropy to express uncertainty in place of standard deviation. This idea was later formalized by Białynicki-Birula and Mycielski in 1975 \cite{bb}, who established an entropic uncertainty relation for position and momentum in quantum mechanics. It has become one of the most important measures in quantum mechanics and information theory,	which is widely used in information theory, communication, mathematics and signal processing\cite{xug}.The Shannon entropy associated with a probability density function $f(t)$ is given by
	\begin{eqnarray} \label{sep}
		\mathcal{E}(f(t))= -\int_{\mathbb{R}^2}f(t)\ln f(t)dt.
	\end{eqnarray}
	
	\begin{lemma}\label{2deup}
		Entropy uncertainty principle for OLCT\cite{huo}: Let $f\in L^2(\mathbb{R})$ and $||f||_2=1$. Then 
		\begin{eqnarray*}
			\mathcal{E}(|f|^2)+\mathcal{E}(|\mathcal{O}_Mf|^2)\geq \ln\pi e.
		\end{eqnarray*}
	\end{lemma}
	\begin{theorem}
		Entropy uncertainty principle for 2D OLCT: Let $f\in L^2(\mathbb{R}^2)$ and $||f||_2=1$. Then
		\begin{eqnarray*}\label{eup1}
			\mathcal{E}(|f|^2)+|b_1b_2|\big(|\mathcal{O}_{M_1M_2}f|^2\big)\geq \ln(\pi e |b_1b_2|^{|b_1b_2|}).
		\end{eqnarray*} 
	\end{theorem}
	\begin{proof}
		Entropy uncertainty principle for 2D FT is given by
		\begin{eqnarray*}
			\mathcal{E}(|f(t)|^2)+\mathcal{E}(|\mathcal{F}f(u)|^2)\geq \ln(\pi e ).
		\end{eqnarray*}
		Taking $u = \frac{u}{b} $, we get
		\begin{eqnarray*}
			\mathcal{E}(|f(t)|^2)+ \mathcal{E}(|\mathcal{F}f(\frac{u}{b})|^2)\geq \ln(\pi e).
		\end{eqnarray*}
		Substituting	$e^{\sum_{k=1}^{2}{\frac{j}{2b_k} \left( a_k t_k^2 + 2 t_k \tau_k \right)}} f(t)$ in place of $f(t)$, we get
		\begin{eqnarray*}
			\mathcal{E}(|e^{\sum_{k=1}^{2}{\frac{j}{2b_k} \left( a_k t_k^2 + 2 t_k \tau_k \right)}} f(t)|^2)+ \mathcal{E}(|\mathcal{F}(e^{\sum_{k=1}^{2}{\frac{j}{2b_k} \left( a_k t_k^2 + 2 t_k \tau_k \right)}} f(t))(\frac{u}{b})|^2)\geq \ln(\pi e).				
		\end{eqnarray*}
		Using relation between 2D FT and 2D OLCT, we get
		\begin{eqnarray*}
			\mathcal{E}(|f(t)|^2)+\mathcal{E}\left(\big|{\sqrt{-b_1b_2}}~e^{-{\sum_{k=1}^{2}}{\frac{i}{2b_k}(-2u_k(d_k\tau_k-b_k\eta_k)+ d_ku_k^2+ d_k\tau_k^2)}}(\mathcal{O}_{M_1M_2} f)(u)\big|^2\right)\geq \ln(\pi e).
		\end{eqnarray*}
		Now, using Definition \eqref{sep} of Shannon entropy, the previous inequality becomes
		\begin{eqnarray*}
			&&\mathcal{E}(|f(t)|^2)-\int_{\mathbb{R}^2}|\sqrt{-b_1b_2} \mathcal{O}_{M_1M_2}f(u)|^2\ln(|\sqrt{-b_1b_2}\mathcal{O}_{M_1M_2}f(u)|^2)du\geq \ln(\pi e).
		\end{eqnarray*}
		Rearranging the inequality, we get	
		\begin{eqnarray*}
			&&\mathcal{E}(|f(t)|^2)-|b_1b_2|\int_{\mathbb{R}^2}|\mathcal{O}_{M_1M_2}f(u)|^2\ln(|\sqrt{-b_1b_2}\mathcal{O}_{M_1M_2}f(u)|^2)du\geq\ln (\pi e).
		\end{eqnarray*}
		We can rewrite the inequality as
		\begin{eqnarray*}
			&&\mathcal{E}(|f(t)|^2)-\{|b_1b_2|\int_{\mathbb{R}^2}|\mathcal{O}_{M_1M_2}f(u)|^2\ln|b_1b_2|du+|b_1b_2|\int_{\mathbb{R}^2}|\mathcal{O}_{M_1M_2}f(u)|^2\ln|\mathcal{O}_{M_1M_2}f(u)|^2du\}\geq ln(\pi e)
		\end{eqnarray*}
		From Definition \eqref{sep} of Shannon entropy, we can rewrite
		\begin{eqnarray*}
			&&\mathcal{E}(|f(t)|^2)+|b_1b_2|\mathcal{E}(|\mathcal{O}_{M_1M_2}f(u)|^2)\geq \ln(\pi e)+|b_1b_2|\int_{\mathbb{R}^2}|\mathcal{O}_{M_1M_2}f(u)|^2\ln|b_1b_2|du.
		\end{eqnarray*}
		Next, using the Parseval's relation, we get 
		\begin{eqnarray*}
			\mathcal{E}(|f(t)|^2)+|b_1b_2|\mathcal{E}(|\mathcal{O}_{M_1M_2}f(u)|^2)\geq \ln(\pi e)+|b_1b_2|\ln|b_1b_2|||f(t)||_2.
		\end{eqnarray*}
		We have,	$||f(t)||_2=1$, it follows that
		\begin{eqnarray*}
			\mathcal{E}(|f(t)|^2)+|b_1b_2|\mathcal{E}(|\mathcal{O}_{M_1M_2}f(u)|^2)&\geq& \ln(\pi e)+\ln|b_1b_2|^{|b_1b_2|}\\&=&\ln(\pi e |b_1b_2|^{|b_1b_2|}). 
		\end{eqnarray*}
		The proof is thus concluded.
	\end{proof}
	\subsection{Nazarov's UP}\ \\
	In 1993, F.L. Nazarov introduced Nazarov’s uncertainty principle \cite{nz}. It describes the behavior when a nonzero function and its Fourier transform remain small outside a compact set. Here we discuss the Nazarov's uncertainty principle with 2D OLCT.
	\begin{theorem} 
		Let	$f\in L^2(\mathbb{R})$ and $ T_1, T_2$ be two finite measurable subsets of $\mathbb{R}^2$ \cite{nz}.Then 
		\begin{eqnarray}\label{narzav2d}
			\int_{\mathbb{R}}|f(t)|^2dt \geq Ce^{(C|T_1||T_2|)}\left(\int_{\mathbb{R}^2\backslash T_1}|f(t)|^2dt+\int_{\mathbb{R}^2\backslash |b| T_2}|\mathcal{O}_{M_1M_2}f(u)|^2du\right).
		\end{eqnarray}
	\end{theorem}
	\begin{proof}
		By Nazarav's {inequality}  for FT in $\mathbb{R}^2$ \cite{nz}, there exists a constant $C > 0$ ensuring that
		\begin{eqnarray*}
			\int_{\mathbb{R}^2}|f(t)|^2dt\leq C e^{C|T_1||T^2|}\left(\int_{\mathbb{R}^2\backslash T_1}|f(t)|^2dt+\int_{\mathbb{R}^2\backslash T_2}|\mathcal{F}f(u)|^2du\right).
		\end{eqnarray*}
		Replacing $u$ by $\frac{u}{b}$, we get
		\begin{eqnarray*}
			\int_{\mathbb{R}^2}|f(t)|^2\leq Ce^{C|T_1||T_2|}\left(\int_{\mathbb{R}^2\backslash T_1}|f(t)|^2dt+\frac{1}{|b_1b_2|}\int_{\mathbb{R}^2\backslash |b| T_2}|\mathcal{F}f(\frac{u}{b})|^2du\right).
		\end{eqnarray*}
		Substituting	$e^{\sum_{k=1}^{2}{\frac{j}{2b_k} \left( a_k t_k^2 + 2 t_k \tau_k \right)}} f(t)$ in place of $f(t)$, we get
		\begin{eqnarray*}
			&&\int_{\mathbb{R}^2}|e^{\sum_{k=1}^{2}{\frac{j}{2b_k} \left( a_k t_k^2 + 2 t_k \tau_k \right)}} f(t)|^2\\&\leq& Ce^{C|T_1||T_2|}\left(\int_{\mathbb{R}^2\backslash T_1}|e^{\sum_{k=1}^{2}{\frac{j}{2b_k} \left( a_k t_k^2 + 2 t_k \tau_k \right)}} f(t)|^2dt+\frac{1}{|b_1b_2|}\int_{\mathbb{R}^2\backslash |b| T_2}|\mathcal{F}e^{\sum_{k=1}^{2}{\frac{j}{2b_k} \left( a_k t_k^2 + 2 t_k \tau_k \right)}} f(t)(\frac{u}{b})|^2du\right).
		\end{eqnarray*}
		Using  the connection between 2D OLCT and 2D FT, we have
		\begin{eqnarray*}
			\int_{\mathbb{R}^2}|f(t)|^2dt&\leq& Ce^{C|T_1||T_2|}\left(\int_{\mathbb{R}^2\backslash T_1}|e^{\sum_{k=1}^{2}{\frac{j}{2b_k} \left( a_k t_k^2 + 2 t_k \tau_k \right)}} f(t)|^2dt\right)+\\&&\left(\frac{1}{|b_1b_2|}\int_{\mathbb{R}^2\backslash |b| T_2}\big|{\sqrt{-b_1b_2}}~e^{-{\sum_{k=1}^{2}}{\frac{i}{2b_k}(-2u_k(d_k\tau_k-b_k\eta_k)+ d_ku_k^2+ d_k\tau_k^2)}}(\mathcal{O}_{M_1M_2} f)(u)\big|^2du\right)\\&\leq&Ce^{C|T_1||T_2|}\left(\int_{\mathbb{R}^2\backslash T_1}|f(t)^2dt +\int_{\mathbb{R}^2\backslash |b| T_2}|\mathcal{O}_{M_1M_2}f(u)|^2du\right).
		\end{eqnarray*}
	\end{proof}
	This completes the proof.
	\subsection{Heisenberg-type uncertainty principle for  2D OLCT}
	\ \\
	The Heisenberg uncertainty principle was first introduced by the German physicist Werner Heisenberg in 1927 \cite{german}.The uncertainty principle states that in quantum systems, precise measurement of one quantity reduces the accuracy with which another can be known. Mathematically, it implies that a function and its Fourier transform cannot both be sharply localized \cite{hupe}.Many researchers have developed various extensions of the Heisenberg uncertainty principle  \cite{pauli}. Here we  discuss about Heisenberg  type uncertainty principle for 2D OLCT.
	\begin{theorem}\label{prop3.6}
		Let $f\in L^2(\mathbb{R}^2)$ be such that $	D_{t_1^mt_2^n}^{M_1M_2 }\in L^2( \mathbb{R}^2)$. Then the following conditions hold for $ k=1,2$
		\begin{eqnarray}\label{4.6}
			\int_{\mathbb{R}^2}|u_k|^2|\mathcal{O}_{M_1M_2}f(u)|^2du\int_{\mathbb{R}^2}|t_k|^2|f(t)|^2dt\geq \bigg|\frac{b_k}{2}\bigg|^2||f||_2^2.
		\end{eqnarray}
	\end{theorem}
	\begin{proof}
		Using Lemma \ref{dup} for $k=1$, we get
		\begin{eqnarray*}
			\int_{\mathbb{R}^2}|u_1|^2|\mathcal{O}_{M_1M_2}f(u)|^2du\int_{\mathbb{R}^2}|t_1|^2|f(t)|^2dt&=&\int_{\mathbb{R}^2}|u_1|^2\bigg| \frac{b_1}{iu_1}\mathcal{O}_{M_1M_2}\big[D_{t_1^1t_2^0}^{M_1M_2}f\big](u_1,u_2)\bigg|^2du\int_{\mathbb{R}^2}|t_1|^2|f(t)|^2dt\\&=&|b_1|^2\int_{\mathbb{R}^2}\bigg| \mathcal{O}_{M_1M_2}\big[D_{t_1^1t_2^0}^{M_1M_2}f\big](u_1,u_2)\bigg|^2du\int_{\mathbb{R}^2}|t_1||f(t)|^2dt
			\\&=&|b_1|^2\int_{\mathbb{R}^2}\bigg|D_{t_1^1t_2^0}^{M_1M_2}f(t)\bigg|^2dt\int_{\mathbb{R}^2}|t_1|^2|f(t)|^2dt\\
			&\geq&|b_1^2|\bigg| \int_{\mathbb{R}^2}t_1f(t)\overline{D_{t_1^1t_2^0}^{M_1M_2}f(t)}dt\bigg|^2\\&\geq&|b_1|^2\bigg| \int_{\mathbb{R}^2}\frac{t_1}{2}[{D_{t_1^1t_2^0}^{M_1M_2}f(t)}\overline{f(t)}+\overline{{D_{t_1^1t_2^0}^{M_1M_2}f(t)}}f(t)]dt\bigg|^2\\
			&\geq&|\frac{|b_1|^2}{4}\bigg|\int_{\mathbb{R}^2}t_1\big(\frac{\partial}{\partial t_1}(f\overline{f})\big)dt\bigg|^2.
		\end{eqnarray*}
		Integrating by parts, we obtain
		\begin{eqnarray}\label{3.55} 
			\int_{\mathbb{R}^2}|u_1|^2|\mathcal{O}_{M_1M_2}(u)|^2\int_{\mathbb{R}^2}|t_1|^2|f(t)|^2dt\geq \frac{b_1^2}{4}||f||_2^2.
		\end{eqnarray}
		Following the similar step $k=2$, we have
		\begin{eqnarray}\label{3.56}
			\int_{\mathbb{R}^2}|u_2|^2|\mathcal{O}_{M_1M_2}(u)|^2\int_{\mathbb{R}^2}|t_2|^2|f(t)|^2dt\geq \frac{b_2^2}{4}||f||_2^2.
		\end{eqnarray}
	\end{proof}
	Combining \eqref{3.55} and \eqref{3.56} we conclude that the proof is complete.  
	\section{Effect of shifting and Scaling Property}
	In this section we discuss the effect of shifting and scaling property on the inequalities and uncertainty principles. Even if frequency shifts disturb the phase, the ability of phase-based reconstruction to resist such shifts helps in reliably recovering the encrypted data\cite{shift}. Let $\alpha$ be a constant and $f(t)=f(\alpha t)$ where $f$ satisfies $tf(t)\in L^2(\mathbb{R}^2)$, $u\mathcal{O}_{M_1M_2}f(u)\in L^2(\mathbb{R}^2). $
	\begin{eqnarray*}
		\mathcal{O}_{M_1M_2}[f(\alpha t)]=\frac{1}{|\alpha|}\mathcal{O}_{M_1'M_2'}(f)(\frac{u}{\alpha})
	\end{eqnarray*} 
	where $M'_1=(\frac{a_1}{\alpha^2},b_1,c_1,d_1\alpha^2,\frac{\tau_1}{\alpha},\alpha \eta_1).$ and  $ M'_2=(\frac{a_2}{\alpha^2},b_2,c_2,d_2\alpha^2,\frac{\tau_2}{\alpha},\alpha \eta_2).$
	\subsection{Effect of Scaling and shifting property on Sharp  Young-Hausdorff inequality for 2D OLCT}\ \\
	Here, we study the influence of the scaling property on the sharp  Young-Hausdorff inequality for the 2D OLCT and discuss its implications on the optimal constant and functional norms. The sharp Young-Hausdorff inequality for 2D OLCT is given \eqref{sharp} as
	
	\begin{eqnarray*}\label{shis}
		&&\left(\int_{\mathbb{R}^2}|\mathcal{O}_{M_1M_2}f(u)|^qdu\right)^{1/q}\leq|b_1b_2|^{\frac{1}{q}-\frac{1}{2}}\left(\frac{p^{1/p}}{q^{1/q}}\right)(2\pi)^{\frac{1}{q}-\frac{1}{p}}\left(\int_{\mathbb{R}^2}|f(t)|^pdt\right)^\frac{1}{p}.
	\end{eqnarray*}
	\subsubsection { \textbf{Scaling effect}}\ \\
	Now using the scaled form $f(t)=\alpha^{1/2}g(\alpha t)$, we have \\
	For L.H.S of \eqref{sharp}
	\begin{eqnarray*}
		\left(\int_{\mathbb{R}^2}|\mathcal{O}_{M_1M_2}(\alpha^{1/2}g(\alpha x))(u)|^qdu\right)^{\frac{1}{q}}&=&\left(\int_{\mathbb{R}^2}|\alpha^{1/2}\mathcal{O}_{M_1M_2}(g(\alpha x))(u)|^qdu\right)^{\frac{1}{q}}\\&=&\left( \int_{\mathbb{R}^2}|\alpha^{1/2}\frac{1}{\alpha}\mathcal{O}_{M'_1M'_2}g(\frac{u}{\alpha})|^qdu\right)^\frac{1}{q}.
	\end{eqnarray*}
	Changing variable $\frac{u}{\alpha}$ by $\omega$, we obtain
	\begin{eqnarray*}
		\left(\int_{\mathbb{R}^2}|\mathcal{O}_{M_1M_2}(\alpha^{1/2}g(\alpha x))(u)|^qdu\right)^{\frac{1}{q}}&=& \left(\frac{1}{|\alpha|^{q/2}} \int_{\mathbb{R}^2}|\mathcal{O}_{M'_1M'_2}g(\omega)|^q\alpha d\omega\right)^\frac{1}{q}\\&=&\left(\frac{1}{|\alpha|^{q/2-1}} \int_{\mathbb{R}^2}|\mathcal{O}_{M'_1M'_2}g(\omega)|^qd\omega\right)^\frac{1}{q}\\
		&=&\frac{1}{|\alpha|^{1/2-1/q}} \left(\int_{\mathbb{R}^2}|\mathcal{O}_{M'_1M'_2}g(\omega)|^qd\omega\right)^\frac{1}{q}\\
	\end{eqnarray*}
	\begin{eqnarray}\label{6.2}
		\left(\int_{\mathbb{R}^2}|\mathcal{O}_{M_1M_2}(\alpha^{1/2}g(\alpha x))(u)|^qdu\right)^{\frac{1}{q}}		
		&=&\frac{1}{|\alpha|^{1/2-1/q}}	|| \mathcal{O}_{M_1,M_2} f ||_{L^q(\mathbb{R}^2)}.
	\end{eqnarray}
	Next, taking $f(t)=\alpha^{1/2}f(\alpha t)$, the R.H.S. the form 
	\begin{eqnarray*}
		\left(\int_{\mathbb{R}^2}|\alpha^{1/2}f(\alpha t)|^pdt\right)^\frac{1}{p}&=&\left(|\alpha|^\frac{p}{2}\int_{\mathbb{R}^2}|f(\alpha t)|^pdt\right)^\frac{1}{p}.
	\end{eqnarray*}
	Changing variable $\alpha t= x$, we obtain
	\begin{eqnarray*}
		&&\left(\int_{\mathbb{R}^2}|\alpha^{1/2}f(\alpha t)|^pdt\right)^\frac{1}{p}=\left(|\alpha|^\frac{p}{2}\int_{\mathbb{R}^2}|f(x)|^pdx\frac{1}{\alpha}\right)^\frac{1}{p}.
	\end{eqnarray*}
	\begin{eqnarray}\label{6.3}
		i.e.	\left(\int_{\mathbb{R}^2}|\alpha^{1/2}f(\alpha t)|^pdt\right)^\frac{1}{p}&=&\alpha^{1/2-1/p}\|f\|_{L^p(\mathbb{R}^2)}.
	\end{eqnarray}
	Combining equations \eqref{6.2} and \eqref{6.3},  we can observe  that there is no effect of scaling property.\\
	\subsubsection{\textbf{ Shifting effect }}\ \\
	
	Let	$f(t)=f(t-\alpha)$ . Then the equation \eqref{shis} becomes 
	\begin{eqnarray*}
		|| \mathcal{O}_{M_1,M_2} f(u) ||_{L^q(\mathbb{R}^2)}&=&\bigg(\int_{\mathbb{R}^2}|(\mathcal{O}_{M_1,M_2} f(t-\alpha)(u))|^qdu\bigg)^\frac{1}{q}\\
		&=&\bigg(\int_{\mathbb{R}^2}\bigg|\mathcal{O}_{M_1M_2}f(t)(u-a\alpha)e^{\frac{i}{2b_1}[2\alpha_1(\tau_1-u_1)-a_1\alpha_1(d_1\tau_1-b_1\eta_1)]}\\&\times&e^{\frac{i}{2b_2}[2\alpha_2(\tau_2-u_2)-a_2\alpha_2(d_2\tau_2-b_2\eta_2)]}\bigg|^qdu\bigg)^\frac{1}{q}\\
		&=&\bigg(\int_{\mathbb{R}^2}|\mathcal{O}_{M_1M_2}f(t)(u-a\alpha)|^qdu\bigg)^\frac{1}{q}.
	\end{eqnarray*}
	Changing variable $u-a\alpha=\xi$
	\begin{eqnarray*}
		&&|| \mathcal{O}_{M_1,M_2} f(u) ||_{L^q(\mathbb{R}^2)}	=\bigg(\int_{\mathbb{R}^2}|\mathcal{O}_{M_1M_2}f(t)(\xi)|^qd\xi \bigg)^\frac{1}{q}\\
		i.e. ,	&&|| \mathcal{O}_{M_1,M_2} f ||_{L^q(\mathbb{R}^2)}=	|| \mathcal{O}_{M_1,M_2} f ||_{L^q(\mathbb{R}^2)}.
	\end{eqnarray*}
	Similarly, for the right-hand side, it can be observed that the shifting effect does not introduce any change
	\begin{remark}
		As we can see there  in no effect on Sharp Young-Hausdorff inequality after shifting the function.
	\end{remark}
	
	\subsection{ For Pitt's inequality}\ \\
	Here we discuss the effect of scaling and shifting  on Pitt’s inequality and how it influences the balance between spatial and frequency domains. In the case of the two-dimensional OLCT, Pitt's inequality is expressed as \eqref{3.44}.
	\begin{eqnarray*}
		\int_{\mathbb{R}^2}|u|^{-\lambda}|\mathcal{O}_{M_1M_2}f(u)|^2du\leq \frac{C_\lambda}{|b_1b_2|^{-\lambda}}\int_{\mathbb{R}^2}|t|^\lambda|f(t)|^2dt.
	\end{eqnarray*}
	
	\subsubsection{\textbf{Scaling effect} }
	
	Let $f(t)=\alpha^{1/2}g(\alpha t)$. Then 
	\begin{eqnarray*}
		&&\int_{\mathbb{R}^2}|u|^{-\lambda}|\mathcal{O}_{M_1M_2}\alpha^{1/2}g(\alpha t)du|^2\leq C_\lambda\int_{\mathbb{R}^2}|t|^{\lambda}|\alpha^{1/2}g(\alpha t)|^2dt.
	\end{eqnarray*}
	Applying Scaling property of 2D OLCT \ref{scaling}, we get
	\begin{eqnarray*}
		&&\int_{\mathbb{R}^2}|u|^{-\lambda}\frac{1}{|\alpha|}|\mathcal{O}_{M'_1M'_2}g(\frac{u}{\alpha})du|^2\leq C_\lambda\int_{\mathbb{R}^2}|t|^\lambda|\alpha||g(\alpha t)|^2dt.
	\end{eqnarray*}
	Changing variable $\frac{u}{\alpha}$ by $\omega$ and $\alpha t = x$
	\begin{eqnarray*}
		&&\int_{\mathbb{R}^2}|u|^{-\lambda}\frac{1}{|\alpha|}|\mathcal{O}_{M'_1M'_2}f(t)(\omega)\alpha d\omega|^2\leq C_\lambda \int_{\mathbb{R}^2}|\frac{x}{\alpha}|^\lambda \alpha|f(x)|^2dx\frac{1}{\alpha}.
	\end{eqnarray*}
	Rearranging the terms, we obtain
	\begin{eqnarray*}
		&&\int_{\mathbb{R}^2}|\alpha \omega|^{-\lambda}|\mathcal{O}_{M'_1M'_2}f(t)(\omega) d\omega|^2\leq C_\lambda \int_{\mathbb{R}^2}|\frac{x}{\alpha}|^\lambda \|f(x)|^2dx.
		\\&&\int_{\mathbb{R}^2}| \omega|^{-\lambda}|\mathcal{O}_{M'_1M'_2}f(t)(\omega) d\omega|^2\leq C_\lambda \int_{\mathbb{R}^2}|x|^\lambda \|f(x)|^2dx.
	\end{eqnarray*}
	\begin{remark}
		It is noteworthy that both Pitt’s inequality and the Hausdorff–Young inequality remain unaffected by the scaling behavior of functions, thereby indicating that these inequalities are invariant under dilation transformations.
	\end{remark}
	\subsubsection{\textbf{Shifting effect}} \ 	Let	$f(t)=f(t-\alpha)$
	\begin{eqnarray*}
		\int_{\mathbb{R}^2}|u|^{-\lambda}|\mathcal{O}_{M_1M_2}f(u)|^2du\leq \frac{C_\lambda}{|b_1b_2|^{-\lambda}}\int_{\mathbb{R}^2}|t|^\lambda|f(t)|^2dt.
	\end{eqnarray*}
	
	\begin{eqnarray*}
		\int_{\mathbb{R}^2}|u|^{-\lambda}|\mathcal{O}_{M_1M_2}f(t-\alpha)(u)|^2du\leq \frac{C_\lambda}{|b_1b_2|^{-\lambda}}\int_{\mathbb{R}^2}|t|^\lambda|f(t-\alpha)|^2dt
	\end{eqnarray*}\begin{eqnarray*}
		\int_{\mathbb{R}^2}|u|^{-\lambda}|\mathcal{O}_{M_1M_2}f(t-\alpha)(u)|^2du&=&\int_{\mathbb{R}^2}|u|^{-\lambda}\bigg|\mathcal{O}_{M_1M_2}f(t)(u-a\alpha)e^{\frac{i}{2b_1}[2\alpha_1(\tau_1-u_1)-a_1\alpha_1(d_1\tau_1-b_1\eta_1)]}\\&\times&e^{\frac{i}{2b_2}[2\alpha_2(\tau_2-u_2)-a_2\alpha_2(d_2\tau_2-b_2\eta_2)]}\bigg|^2du\\
		&=&\int_{\mathbb{R}^2}|u|^{-\lambda}|\mathcal{O}_{M_1M_2}f(t)(u-a\alpha)|^2du.
	\end{eqnarray*}
	Changing variable $u-a\alpha=\xi$
	\begin{eqnarray*}
		\int_{\mathbb{R}^2}|u|^{-\lambda}|\mathcal{O}_{M_1M_2}f(t-\alpha)(u)|^2du	&=&\int_{\mathbb{R}^2}|\xi+a\alpha|^{-\lambda}|\mathcal{O}_{M_1M_2}f(t)(\xi)|^2d\xi.
	\end{eqnarray*}
	{After shifting the function we can see that Although shifting a function only changes the phase of its OLCT, Pitt’s inequality is still affected. This is because the weight functions like $|u|^{-\lambda}$ and $|t|^{\lambda}$
		are not shift-invariant. So, even though the OLCT magnitude shifts, the inequality does not remain exactly the same. }
	
	\subsection{For Entropy Uncertainty principle }\ \\
	Here we discuss the effect of scaling on the entropy inequality and its influence on the distribution of a signal in both domains. Using  Theorem \ref{2deup}
	\begin{eqnarray}\label{5.41}
		\mathcal{E}(|f|^2)+|b_1b_2|\mathcal{E}\big(|\mathcal{O}_{M_1M_2}f|^2\big)\geq \ln(\pi e |b_1b_2|^{|b_1b_2|})
	\end{eqnarray} 
	
	\subsubsection{\textbf{Scaling effect}} Let consider 
	
	\begin{eqnarray*}
		f(t)=\alpha^{\frac{1}{2}}f(\alpha t).
	\end{eqnarray*}  
	Putting $f(t)$ in the 1st term  of \eqref{5.41}, we have
	\begin{eqnarray*}\\ \mathcal{E}(|\alpha^{\frac{1}{2}}f(\alpha   t)|^2)&=&-\int_{\mathbb{R}^2}|\alpha^{\frac{1}{2}}f(\alpha t)|^2\ln(|\alpha^{\frac{1}{2}}f(\alpha t)|^2)dt.
	\end{eqnarray*}
	Changing variable $\alpha t= x$, we obtain
	\begin{eqnarray*}
		\mathcal{E}(|\alpha^{\frac{1}{2}}f(\alpha t)|^2)	&=&-\int_{\mathbb{R}^2}|\alpha^{\frac{1}{2}}f(x)|^2\ln(|\alpha^{\frac{1}{2}}f(x)|^2)dx\frac{1}{\alpha}\\
		&=&-\int_{\mathbb{R}^2}|f(x)|^2\ln(|\alpha^{\frac{1}{2}}f(x)|^2)dx\\
		&=&-\int_{\mathbb{R}^2}|f(x)|^2\bigg(\ln(|\alpha^{\frac{1}{2}}|^2)+\ln(|f(x)|^2)\bigg)dx\\
		&=&-\int_{\mathbb{R}^2}|f(x)|^2(\ln(|\alpha^{\frac{1}{2}}|^2)-\int_{\mathbb{R}^2}|f(x)|^2\bigg(\ln(|f(x)|^2)\bigg)dx.
	\end{eqnarray*}
	We have,	$||f||=1$, then
	\begin{eqnarray}\label{5.1}
		\mathcal{E}(|\alpha^{\frac{1}{2}}f(\alpha t)|^2)	&=&-\ln\alpha+\mathcal{E}(|f|^2).
	\end{eqnarray} 
	Next, for 2nd term of \eqref{5.41}, we get
	\begin{eqnarray*}
		\mathcal{E}\big(|\mathcal{O}_{M_1M_2}\alpha^{\frac{1}{2}}f(\alpha t)(u)|^2\big)&=&-\int_{\mathbb{R}^2}|\mathcal{O}_{M_1M_2}\alpha^{\frac{1}{2}}f(\alpha t)(u)|^2\ln|\mathcal{O}_{M_1M_2}\alpha^{\frac{1}{2}}f(\alpha t)(u)|^2du\\&=&-\int_{\mathbb{R}^2}|\alpha||\frac{1}{\alpha}\mathcal{O}_{M'_1M'_2}f(\frac{u}{\alpha})|^2\ln\bigg(|\alpha||\frac{1}{\alpha}\mathcal{O}_{M'_1M'_2}f(\frac{u}{\alpha})|^2\bigg)du\\
		&=&-\int_{\mathbb{R}^2}|\alpha||\frac{1}{\alpha}\mathcal{O}_{M'_1M'_2}f(\omega)|^2\ln\bigg(|\alpha||\frac{1}{\alpha}\mathcal{O}_{M'_1M'_2}f(\omega)|^2\bigg)\alpha d\omega\\
		&=&-\int_{\mathbb{R}^2}|\mathcal{O}_{M'_1M'_2}f(\omega)|^2\ln\bigg(\frac{1}{|\alpha|}|\mathcal{O}_{M'_1M'_2}f(\omega)|^2\bigg)d\omega\\
		&=&-\int_{\mathbb{R}^2}|\mathcal{O}_{M'_1M'_2}f(\omega)|^2\ln\frac{1}{|\alpha|}-\int_{\mathbb{R}^2}|\mathcal{O}_{M'_1M'_2}f(\omega)|^2\ln |\mathcal{O}_{M'_1M'_2}f(\omega)|^2d\omega\\
	\end{eqnarray*}
	From the Definition\eqref{sep} of Shannon entropy, we get
	\begin{eqnarray}\label{5.2}
		\mathcal{E}\big(|\mathcal{O}_{M_1M_2}\alpha^{\frac{1}{2}}f(\alpha t)(u)|^2\big)	&=&-\ln\frac{1}{|\alpha|}+\mathcal{E}\bigg(|\mathcal{O}_{M'_1M'_2}f(\omega)|^2\bigg).
	\end{eqnarray}
	Substitute the equations  \eqref{5.1} and \eqref{5.2}  using Theorem \ref{eup}, we get 
	\begin{eqnarray*}
		&&\mathcal{E}(|\alpha^{\frac{1}{2}}f(\alpha t)|^2)+	|b_1b_2|\mathcal{E}\big(|\mathcal{O}_{M_1M_2}\alpha^{\frac{1}{2}}f(\alpha t)(u)|^2\big)\\
		&=&-\ln\alpha+\mathcal{E}(|f|^2)+|b_1b_2|\bigg(-\ln\frac{1}{|\alpha|}+\mathcal{E}\bigg(|\mathcal{O}_{M'_1M'_2}f(\omega)|^2\bigg)\bigg)\\
		&=&-\ln\alpha+\mathcal{E}(|f|^2)+|b_1b_2|\bigg(\ln{\alpha}+\mathcal{E}\bigg(|\mathcal{O}_{M'_1M'_2}f(\omega)|^2\bigg)\bigg)\\
		&=&\mathcal{E}(|f|^2)+|b_1b_2|\bigg(\mathcal{E}\bigg(|\mathcal{O}_{M'_1M'_2}f(\omega)|^2\bigg)\bigg)-\ln\alpha+|b_1b_2|(\ln{\alpha})\\
		&\geq&\ln(\pi e |b_1b_2|^{|b_1b_2|})+(\ln{\alpha}^{|b_1b_2|-1}).
	\end{eqnarray*}
	As we can see after scaling the function, we get an extra constant term, which may increase the lower bound.
	\\
	\subsubsection{\textbf{Effect of shifting property on entropy uncertainty principle} }
	Shifting a signal affects its position but leaves the entropy measures unchanged, preserving the uncertainty relationship.
	
	Let	$f(t)=f(t-\alpha)$ in equation \ref{5.41} of 1st term, we have 
	\begin{eqnarray*} 
		\mathcal{E}(|f(t-\alpha )|^2)&=&-\int_{\mathbb{R}^2}|f(t-\alpha )|^2\ln(|f(t-\alpha )|^2)dt.
	\end{eqnarray*}
	Changing variable $t-\alpha=x$, we obtain
	\begin{eqnarray*}
		\mathcal{E}(|f(t-\alpha )|^2)&=&-\int_{\mathbb{R}^2}|f(x )|^2\ln(|f(x )|^2)dx.
	\end{eqnarray*}
	Next, for 2nd part of \eqref{5.41}, we get
	\begin{eqnarray*}
		\mathcal{E}\big(|\mathcal{O}_{M_1M_2}f(t-\alpha )(u)|^2\big)&=&-\int_{\mathbb{R}^2}|\mathcal{O}_{M_1M_2}f(t-\alpha)(u)|^2\ln|\mathcal{O}_{M_1M_2}f(t-\alpha )(u)|^2du\\&=&\int_{\mathbb{R}^2}\bigg|\mathcal{O}_{M_1M_2}f(t)(u-a\alpha)e^{\frac{i}{2b_1}[2\alpha_1(\tau_1-u_1)-a_1\alpha_1(d_1\tau_1-b_1\eta_1)]}\\&\times&e^{\frac{i}{2b_2}[2\alpha_2(\tau_2-u_2)-a_2\alpha_2(d_2\tau_2-b_2\eta_2)]}\bigg|^2\\&\times&\ln\bigg|\mathcal{O}_{M_1M_2}f(t)(u-a\alpha)e^{\frac{i}{2b_1}[2\alpha_1(\tau_1-u_1)-a_1\alpha_1(d_1\tau_1-b_1\eta_1)]}\\&\times&e^{\frac{i}{2b_2}[2\alpha_2(\tau_2-u_2)-a_2\alpha_2(d_2\tau_2-b_2\eta_2)]}\bigg|^2du\\
		&=&\int_{\mathbb{R}^2}|\mathcal{O}_{M_1M_2}f(t)(u-a\alpha)|^2\ln |\mathcal{O}_{M_1M_2}f(t)(u-a\alpha)|^2du.
	\end{eqnarray*}
	Changing variable $u-a\alpha=\xi$
	\begin{eqnarray*}
		&&\mathcal{E}\big(|\mathcal{O}_{M_1M_2}f(t-\alpha )(u)|^2\big)=\int_{\mathbb{R}^2}|\mathcal{O}_{M_1M_2}f(t)(\xi)|^2\ln |\mathcal{O}_{M_1M_2}f(t)(\xi)|^2du. 
	\end{eqnarray*}
	From the Definition\eqref{sep} of Shannon entropy, we can rewrite as
	\begin{eqnarray*}
		&&\mathcal{E}\big(|\mathcal{O}_{M_1M_2}f(t-\alpha )(u)|^2\big)=\mathcal{E}\big(|\mathcal{O}_{M_1M_2}f(t )(u)|^2\big).
	\end{eqnarray*}
	Therefore, the entropy-based uncertainty principle associated with the OLCT remains unaffected by time shifting of the input signal.
	\subsection{ For Heisenberg type uncertainty principle}\ \\
	The effect of scaling and shifting plays a significant role in the behavior of the uncertainty principle under the OLCT. We examine how these operations impact time-frequency localization. In Theorem \ref{prop3.6}, we have
	\begin{eqnarray}\label{heisen}
		\int_{\mathbb{R}^2}|u_k|^2|\mathcal{O}_{M_1M_2}f(u)|^2du\int_{\mathbb{R}^2}|t_k|^2|f(t)|^2dt\geq \bigg|\frac{b_k}{2}\bigg|^2||f||_2^2.
	\end{eqnarray}
	\subsubsection{\textbf{Scaling effect}}
	
	Let $f(t)=f(\alpha t_1,t_2)$. Then using $f(t)$, we have
	in \eqref{heisen}
	\begin{eqnarray}
		\mathcal{O}_{M_1M_2}f(\alpha t_1,t_2)(u)=\frac{1}{\alpha}\mathcal{O}_{M_1'M_2}f(t)(\frac{u_1}{\alpha}, u_2),
	\end{eqnarray}
	where 	 $M'_1=(\frac{a_1}{\alpha^2},b_1,c_1,d_1\alpha^2,\frac{\tau_1}{\alpha},\alpha \eta_1).$ and $M_2 =(a_2,b_2,c_2,d_2,\tau_2,\eta_2).$\\ For $k=1$
	\begin{eqnarray*}
		\int_{\mathbb{R}^2}|u_1|^2|\mathcal{O}_{M_1M_2}f(\alpha t_1,t_2)(u)|^2du=\int_{\mathbb{R}^2}|u_1|^2|\frac{1}{\alpha}\mathcal{O}_{M_1'M_2}f(t)(\frac{u_1}{\alpha}, u_2)|^2du.
	\end{eqnarray*}
	Substituting	$\frac{u_1}{\alpha}= \xi_1$ and $u_2= \xi_2$.
	\begin{eqnarray*}
		\int_{\mathbb{R}^2}|u_1|^2|\mathcal{O}_{M_1M_2}f(\alpha t_1,t_2)(u)|^2du&=&\int_{\mathbb{R}^2}|\alpha \xi|^2|\frac{1}{\alpha}\mathcal{O}_{M_1'M_2}f(t)(\xi_1,\xi_2)|^2\alpha d\xi
	\end{eqnarray*}
	Simplifying the equation, we get
	\begin{eqnarray}\label{5.4}
		\int_{\mathbb{R}^2}|u_1|^2|\mathcal{O}_{M_1M_2}f(\alpha t_1,t_2)(u)|^2du	&=&\alpha^2\int_{\mathbb{R}^2}|\xi|^2|\mathcal{O}_{M_1'M_2}f(t)(\xi_1,\xi_2)|d\xi
	\end{eqnarray}
	\begin{eqnarray*}
		\int_{\mathbb{R}^2}|t_1|^2|f(t)|^2=\int_{\mathbb{R}^2}|t_1|^2|f(t_1,t_2)|^2dt
	\end{eqnarray*}
	Let $f(t)=f(\alpha t_1,t_2)$. Then
	\begin{eqnarray*}
		\int_{\mathbb{R}^2}|t_1|^2|f(t)|^2=\int_{\mathbb{R}^2}|\alpha t_1|^2|f(\alpha t_1,t_2)|^2dt.
	\end{eqnarray*}
	Changing variable $\alpha t_1=x_1$ and $t_2=x_2$
	\begin{eqnarray}\label{5.5}
		\int_{\mathbb{R}^2}|t_1|^2|f(t)|^2=\frac{1}{\alpha^2}\int_{\mathbb{R}^2}|x_1|^2|f(x_1,x_2)|^2dt.
	\end{eqnarray}
	Combing equation \eqref{5.4} and \eqref{5.5}, we can see , there is no effect after scaling.\\ Similarly we can  see for $ k=2$.
	\subsubsection{\textbf{Shifting effect}}
	
	Let	$f(t_1,t_2)=f(t_1-\alpha_1,t_2)$ in inequality \eqref{heisen}, we get
	for $ k=1$
	\begin{eqnarray*}
		\int_{\mathbb{R}^2}|u_1|^2|\mathcal{O}_{M_1M_2}f(t_1-\alpha_1,t_2)(u)|^2du&=&\int_{\mathbb{R}^2}|u_1|^2|\mathcal{O}_{M_1M_2}f(t)(u_1-a_1\alpha_1,u_2)e^{\frac{i}{2b_1}[2\alpha_1(\tau_1-u_1)-a_1\alpha_1(d_1\tau_1-b_1\eta_1)]}|^2du\\
		&=&\int_{\mathbb{R}^2}|u_1|^2|\mathcal{O}_{M_1M_2}f(t)(u_1-a_1\alpha_1,u_2)|^2du.
	\end{eqnarray*}
	Taking	$u_1-a_1\alpha_1=\omega_1, u_2=\omega_2$, we obtain
	\begin{eqnarray*}
		\int_{\mathbb{R}^2}|u_1|^2|\mathcal{O}_{M_1M_2}f(t_1-\alpha_1,t_2)(u)|^2du&=&\int_{\mathbb{R}^2}|\omega_1+a_1\alpha_1|^2|\mathcal{O}_{M_1M_2}f(t)(\omega_1,\omega_2)|^2d\omega\\
		&=&\int_{\mathbb{R}^2}|\omega_1|^2|\mathcal{O}_{M_1M_2}f(t_1,t_2)(\omega)|^2d\omega
	\end{eqnarray*}
	Let	$f(t)=f(t-\alpha)$. Then
	\begin{eqnarray*}
		\int_{\mathbb{R}^2}|t_1|^2|f(t-\alpha)|^2dt&=&\int_{\mathbb{R}^2}|x_1+\alpha_1 |^2|f(x)|^2dx\\
		&=&\int_{\mathbb{R}^2}|x_1 |^2|f(x)|^2dx.
	\end{eqnarray*}
	So the expression is unaffected after applying the shifting property.
	\section{Numerical simulation}\ \\
	In this section we use the following examples to illustrate our theoretical results.  Let the Gaussian function $f(t)=e^{-\alpha t^2}$. Then from the Definition \ref{olct}, we get \begin{eqnarray*}
		\mathcal{O}_{M_1M_2}(e^{-\alpha t^2}))(u)&=& \int_{\mathbb{R}^2} \mathcal{K}_{M_1}(u_1,t_1) \mathcal{K}_{M_2}(u_2,t_2)e^{-\alpha_r t_r^2} dt\\
		&=& \frac{1}{(\sqrt{2 \pi})^2 \sqrt{-b_1b_2}}\int_{\mathbb{R}^2}e^{{\sum_{r=1}^{2}}{\frac{i}{2b_r}(a_r t_r^2 + 2 t_r (\tau_r-u_r)-2u_r(d_r\tau_r-b_r\eta_r)+ d_ru_r^2+ d_r\tau_r^2)}}e^{-\alpha_r t_r^2}dt\\
		&=& \frac{1}{({2 \pi}) \sqrt{-b_1b_2}}e^{{\sum_{r=1}^{2}}{\frac{i}{2b_r}(-2u_r(d_r\tau_r-b_r\eta_r)+ d_ru_r^2+ d_r\tau_r^2)}}\int_{\mathbb{R}^2}e^{{\sum_{r=1}^{2}}{\frac{i}{2b_r}(a_r t_r^2 + 2 t_r (\tau_r-u_r))}}e^{-\alpha_r t_r^2}dt\\ 
		&=&\frac{1}{({2 \pi}) \sqrt{-b_1b_2}}e^{{\sum_{r=1}^{2}}{\frac{i}{2b_r}(-2u_r(d_r\tau_r-b_r\eta_r)+ d_ru_r^2+ d_r\tau_r^2)}}\int_{\mathbb{R}^2}e^{{\sum_{r=1}^{2}}{\frac{i}{2b_r}(t_r^2(a_r+2ib_r\alpha_r)+ 2 t_r (\tau_r-u_r))}}dt\\ 
		&=& \frac{1}{({2 \pi}) \sqrt{-b_1b_2}}e^{{\sum_{r=1}^{2}}{\frac{i}{2b_r}(-2u_r(d_r\tau_r-b_r\eta_r)+ d_ru_r^2+ d_r\tau_r^2)}}\\
		&\times&\int_{\mathbb{R}^2}e^{{\sum_{r=1}^{2}}{\frac{i}{2b_r}{(a_r+2ib_r\alpha_r)\big[(t_r^2+ 2 t_r \frac{(\tau_r-u_r)}{(a_r+2ib_r\alpha_r)}+\big(\frac{(\tau_r-u_r)}{(a_r+2ib_r\alpha_r)}\big)^2-\big(\frac{(\tau_r-u_r)}{(a_r+2ib_r\alpha_r)}\big)^2}\big])}}dt\\
		&=&\frac{1}{({2 \pi}) \sqrt{-b_1b_2}}e^{{\sum_{r=1}^{2}}{\frac{i}{2b_r}(-2u_r(d_r\tau_r-b_r\eta_r)+ d_ru_r^2+ d_r\tau_r^2)}}\\
		&\times&\int_{\mathbb{R}^2}e^{{\sum_{r=1}^{2}}{\frac{i}{2b_r}{(a_r+2ib_r\alpha_r)\big[\big(t_r +(\frac{(\tau_r-u_r)}{(a_r+2ib_r\alpha_r)}\big)^2-\big(\frac{(\tau_r-u_r)}{(a_r+2ib_r\alpha_r)}\big)^2}\big])}}dt\\ 
		&=&\frac{1}{({2 \pi}) \sqrt{-b_1b_2}}e^{{\sum_{r=1}^{2}}{\frac{i}{2b_r}(-2u_r(d_r\tau_r-b_r\eta_r)+ d_ru_r^2+ d_r\tau_r^2)}}e^{{\sum_{r=1}^{2}}{\frac{-i}{2b_r}(a_r+2ib_r\alpha_r)\big(\frac{(\tau_r-u_r)}{(a_r+2ib_r\alpha_r)}\big)^2}}\\
		&\times&\int_{\mathbb{R}^2}e^{{\sum_{r=1}^{2}}{\frac{i}{2b_r}{(a_r+2ib_r\alpha_r)\big(t_r +\frac{(\tau_r-u_r)}{(a_r+2ib_r\alpha_r)}\big)^2}}}dt.
	\end{eqnarray*} 
	After simplifying, we get 
	\begin{eqnarray*}
		\mathcal{O}_{M_1M_2}(e^{-\alpha t^2}))(u)&=& \frac{1}{({2 \pi}) \sqrt{-b_1b_2}}e^{{\sum_{r=1}^{2}}{\frac{i}{2b_r}(-2u_r(d_r\tau_r-b_r\eta_r)+ d_ru_r^2+ d_r\tau_r^2)}}e^{{\sum_{r=1}^{2}}{\frac{-i}{2b_r}(a_r+2ib_r\alpha_r)\big(\frac{(\tau_r-u_r)}{(a_r+2ib_r\alpha_r)}\big)^2}}\\&\times&\sqrt{\frac{\pi}{\frac{-i(a_1+2i\alpha_1b_1)}{2b_1}}} \sqrt{\frac{\pi}{\frac{-i(a_2+2i\alpha_2b_2)}{2b_2}}}.
	\end{eqnarray*}
	\begin{eqnarray*} 
		&&\mathcal{O}_{M_1M_2}(e^{-\alpha t^2}))(u)\\&=&e^{{\sum_{r=1}^{2}}{\frac{i}{2b_r}(-2u_r(d_r\tau_r-b_r\eta_r)+ d_ru_r^2+ d_r\tau_r^2)}}e^{{\sum_{r=1}^{2}}{\frac{-i}{2b_r}\big(\frac{(\tau_r-u_r)^2}{(a_r+2ib_r\alpha_r)}\big)}}\sqrt{\frac{1}{a_1+2i\alpha_1b_1}}\sqrt{\frac{1}{a_2+2i\alpha_2b_2}} .
	\end{eqnarray*} 
	After solving and rearranging the expression, we have\\
	$\mathcal{O}_{M_1M_2}(e^{-\alpha t^2}))(u)$
	\begin{eqnarray}\label{6.1}
		&=&e^{{\sum_{r=1}^{2}}{\frac{i}{2b_r}(-2u_r(d_r\tau_r-b_r\eta_r)+ d_ru_r^2+ d_r\tau_r^2)}}e^{{\sum_{r=1}^{2}}{\frac{-1}{2b_r}\big(\frac{(\tau_r-u_r)^2(2b_r\alpha_r+a_ri)}{(4b_r^2\alpha^2_r+a_r^2)}\big)}}\sqrt{\frac{1}{a_1+2i\alpha_1b_1}}\sqrt{\frac{1}{a_2+2i\alpha_2b_2}}.			
	\end{eqnarray}
	\subsection{Heisenberg type Uncertainty principle }\ \\
	We now explore the Heisenberg-type uncertainty principle in the 2D OLCT setting using a Gaussian function as an example. From theorem \ref{prop3.6}, we have
	\begin{eqnarray*}
		\int_{\mathbb{R}^2}|u_k|^2|\mathcal{O}_{M_1M_2}f(u)|^2du\int_{\mathbb{R}^2}|t_k|^2|f(t)|^2dt\geq \bigg|\frac{b_k}{2}\bigg|^2||f||_2^2.
	\end{eqnarray*}
	For $k=1$
	\begin{eqnarray*}
		\int_{\mathbb{R}^2}|u_1|^2|\mathcal{O}_{M_1M_2}f(u)|^2du\int_{\mathbb{R}^2}|t_1|^2|f(t)|^2dt\geq \bigg|\frac{b_1}{2}\bigg|^2||f||_2^2.
	\end{eqnarray*}
	Using the equation \eqref{6.1} we get
	\begin{eqnarray*}
		\int_{\mathbb{R}^2}|u_1|^2\big|e^{{\sum_{r=1}^{2}}{\frac{-1}{2b_r}\big(\frac{(\tau_r-u_r)^2(2b_r)}{(4b^2_r\alpha^2_r+a^2_r)}\big)}}\sqrt{\frac{1}{a_1+2i\alpha_1b_1}}\sqrt{\frac{1}{a_2+2i\alpha_2b_2}}|^2du\int_{\mathbb{R}^2}|t_1|^2|e^{-\alpha t^2}|^2dt\geq \bigg|\frac{b_1}{2}\bigg|^2||e^{-\alpha t^2}||_2^2.
	\end{eqnarray*}
	Rearranging and solving, we get
	\begin{eqnarray*}
		\sqrt{\frac{1}{a_1^2+4\alpha_1^2b^2_1}}\sqrt{\frac{1}{a_2^2+4\alpha_2^2b_2^2}}	\int_{\mathbb{R}^2}|u_1|^2\big|e^{{\sum_{r=1}^{2}}{\frac{-1}{2b_r}\big(\frac{(\tau_r-u_r)^2(2b_r)}{(4b^2_r\alpha^2_r+a^2_r)}\big)}}|^2du\int_{\mathbb{R}^2}|t_1|^2|e^{-\alpha t^2}|^2dt\geq \bigg|\frac{b_1}{2}\bigg|^2||e^{-\alpha t^2}||_2^2.
	\end{eqnarray*}
	
	\begin{table}[h]  
		\caption{Numerical simulation of Gaussian function for Theorem \ref{prop3.6} } 
		
		\begin{tabular}{|c|c|c|c|c|c|c|}  
			\hline                         
			$\alpha_1$ 	&\textbf{$b_1$} & \textbf{LHS} &\textbf{RHS} &\textbf{Difference} \\  
			\hline                    
			1.5& 1.1    &   1.5564 &   1.21&   0.346397 \\  
			\hline                        
			& 1.3  &  2.12188 &   1.69 &   0.431884\\  
			\hline 
			&1.5     & 2.78162    &2.25  &  0.531618\\
			\hline 
			& 1.7    &3.5356&    2.89&     0.645601\\
			\hline
			& 1.9      &4.38383    &3.61  & 0.773831 \\
			\hline
		\end{tabular}\\
		\label{tab:example}-\\
		\begin{tabular}{|c|c|c|c|c|c|c|}  
			\hline                         
			$\alpha_1$ 	&\textbf{$b_1$} & \textbf{LHS} &\textbf{RHS} &\textbf{Difference} \\  
			\hline                    
			2& 1.1    &   1.99884 &   1.21& 0.788838  \\  
			\hline                        
			& 1.3  & 2.75282 &   1.69 &1.06282  \\  
			\hline 
			&1.5     & 3.63247   &2.25  & 1.38247\\
			\hline 
			& 1.7    &4.63778&    2.89&     1.74778\\
			\hline
			& 1.9      &5.76875    &3.61  & 2.15875 \\
			\hline
		\end{tabular}\\
		-\\
		\begin{tabular}{|c|c|c|c|c|c|c|}  
			\hline                         
			$\alpha_1$ 	&\textbf{$b_1$} & \textbf{LHS} &\textbf{RHS} &\textbf{Difference} \\  
			\hline                    
			2.5& 1.1    & 2.45437  &   1.21&  1.24437  \\  
			\hline                        
			& 1.3  &  3.39685 &   1.69 & 1.70685  \\  
			\hline 
			&1.5     & 4.4964    &2.25  &  2.2464\\
			\hline 
			& 1.7    & 5.75304&    2.89& 2.86304    \\
			\hline
			& 1.9      &7.16676   &3.61  & 3.55676 \\
			\hline
		\end{tabular}
	\end{table}
	\begin{figure}
		\includegraphics[width=5.5 cm] {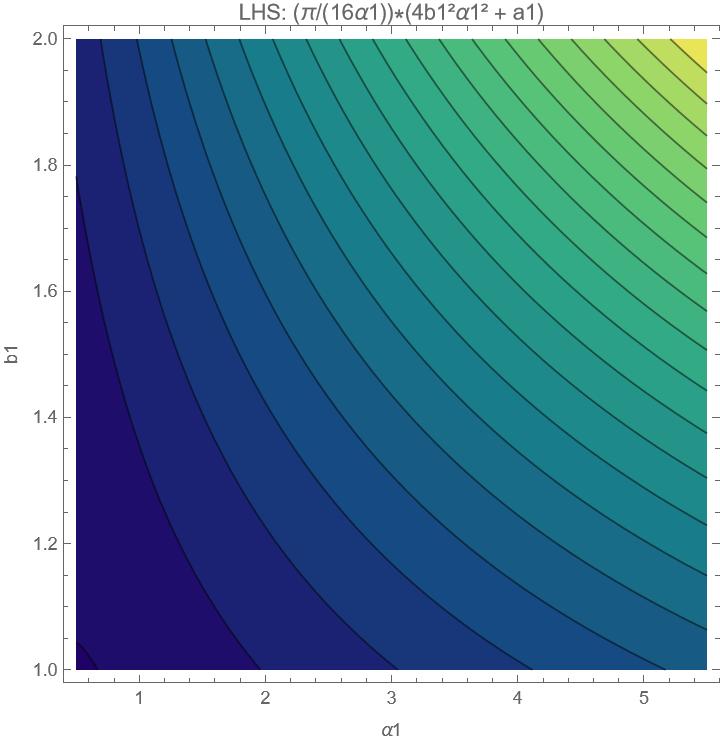}\label{graph1}
		\caption{RHS of Heisenberg type uncertainty principle (Theorem \ref{prop3.6})}.
		\includegraphics[width=5.5 cm] {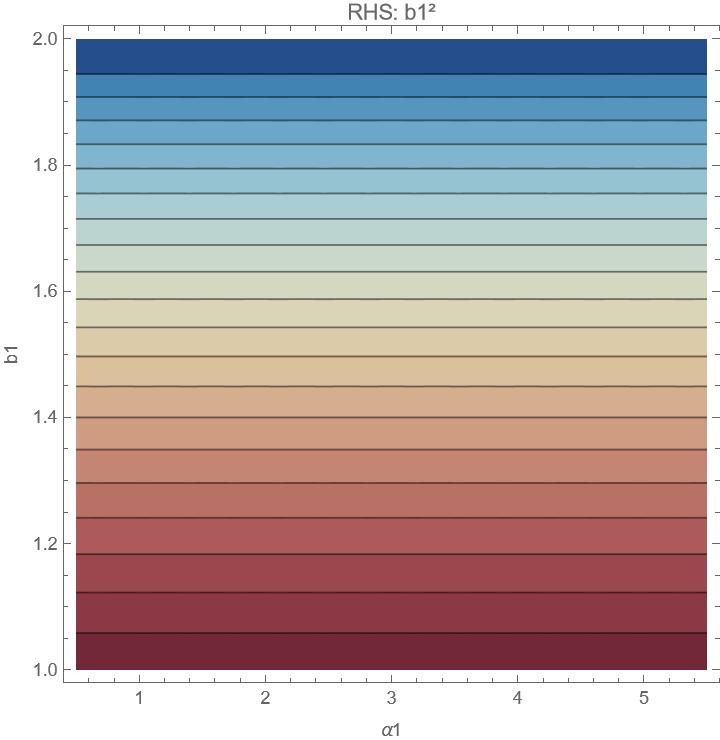}\label{graph2}
		\caption{LHS of Heisenberg type uncertainty principle (Theorem \ref{prop3.6}) }.
		\includegraphics[width=7 cm] {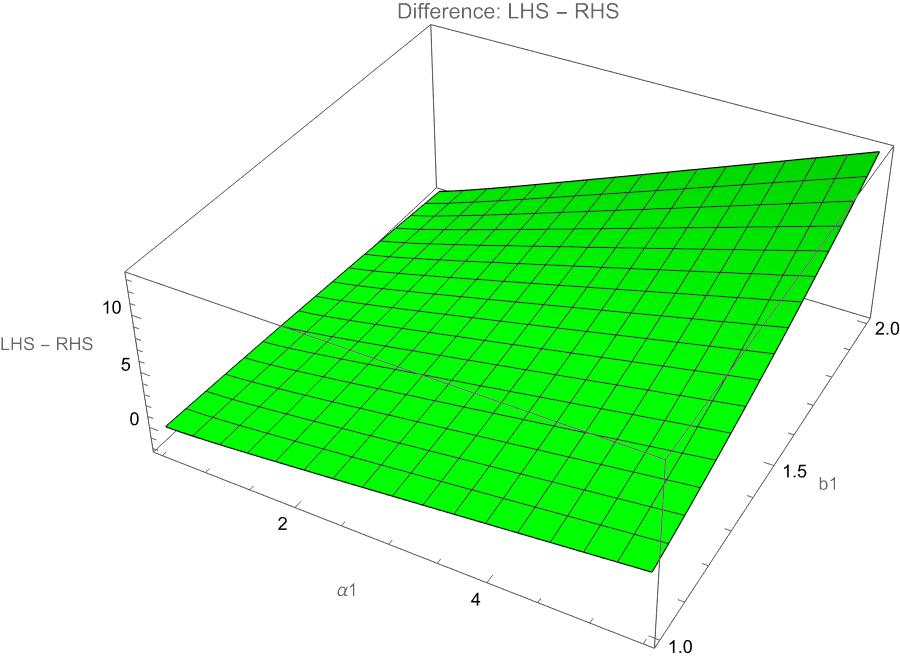}\label{graph3}
		\caption{Comparison of LHS and RHS functions for different parameter values (Theorem \ref{prop3.6})}.
	\end{figure}
	\begin{remark}	
		In Table \ref{tab:example} it is showing for $\alpha_1=1.5$ when $b_1$ is near 1, the difference between LHS and RHS is reducing. Same kind of pattern we are getting for $\alpha_1=2,2.5$. Also for the value $b_1$ near 1, if we are reducing the value $\alpha_1$ from 2.5 to 1.5 the difference of LHS and RHS is decreasing. Its graphical analysis also we can see from  Fig.$1,2$ and $3$.
	\end{remark}
	\subsection{ Sharp Young-Hausdorff  inequality }
	Now we explore Sharp young Hausdorff inequality in 2D OLCT using Gaussian function. The Sharp Young- Hausdorff inequality is  \eqref{sharp}
	\begin{eqnarray*}
		|| \mathcal{O}_{M_1,M_2} f ||_{L^q(\mathbb{R}^2)} \leq \mathcal{K} \|f\|_{L^p(\mathbb{R}^2)},
	\end{eqnarray*}
	where \begin{eqnarray*}
		\mathcal{K} = |b_1 b_2|^{\frac{1}{q} - \frac{1}{2}} \left( \frac{p^{1/p}}{q^{1/q}} \right) (2\pi)^{\frac{1}{q} - \frac{1}{p}}.
	\end{eqnarray*}
	Next, using equation \eqref{6.1}, we get
	\begin{eqnarray*}
		&&\left(\int_{\mathbb{R}^2}|\mathcal{O}_{M_1M_2}(e^{-\alpha t^2})(u)|^qdu\right)^{1/q}=\left(\int_{\mathbb{R}^2} \big|e^{{\sum_{r=1}^{2}}{{-1}\big(\frac{(\tau_r-u_r)^2}{(4b^2_r\alpha^2_r+a^2_r)}\big)}}\sqrt{\frac{1}{a_1+2i\alpha_1b_1}}\sqrt{\frac{1}{a_2+2i\alpha_2b_2}}|^q du\right)^{1/q}
	\end{eqnarray*}
	\begin{eqnarray*}
		\mathcal{K} \|f\|_{L^p(\mathbb{R}^2)}=|b_1 b_2|^{\frac{1}{q} - \frac{1}{2}} \left( \frac{p^{1/p}}{q^{1/q}} \right) (2\pi)^{\frac{1}{q} - \frac{1}{p}}\left(\int_{\mathbb{R}^2}|e^{-\alpha t^2}|^pdt\right)^{1/p}.
	\end{eqnarray*}
	
	\begin{figure}
		\includegraphics[width=5.5 cm]{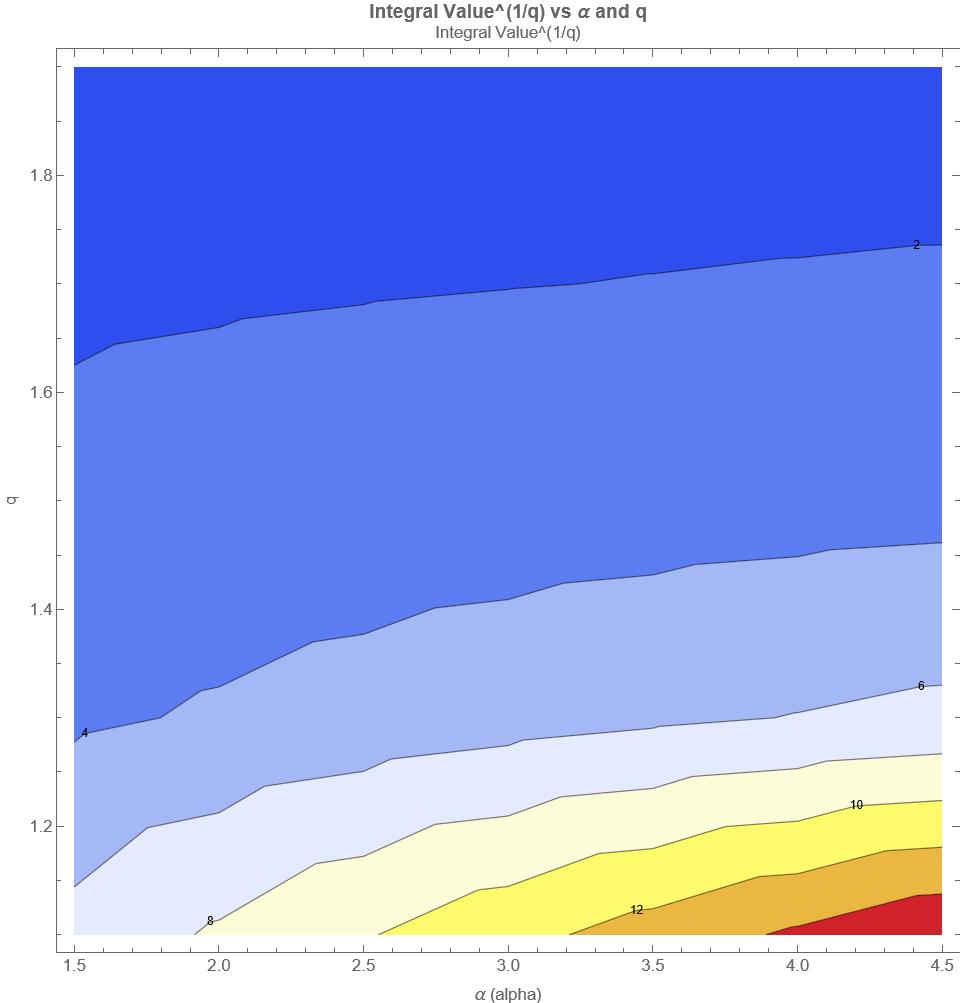}	 
		\caption{LHS of Sharp Hausdorff Young inequality (Theorem \ref{sharp})}
		
		\includegraphics[width=5.5 cm] {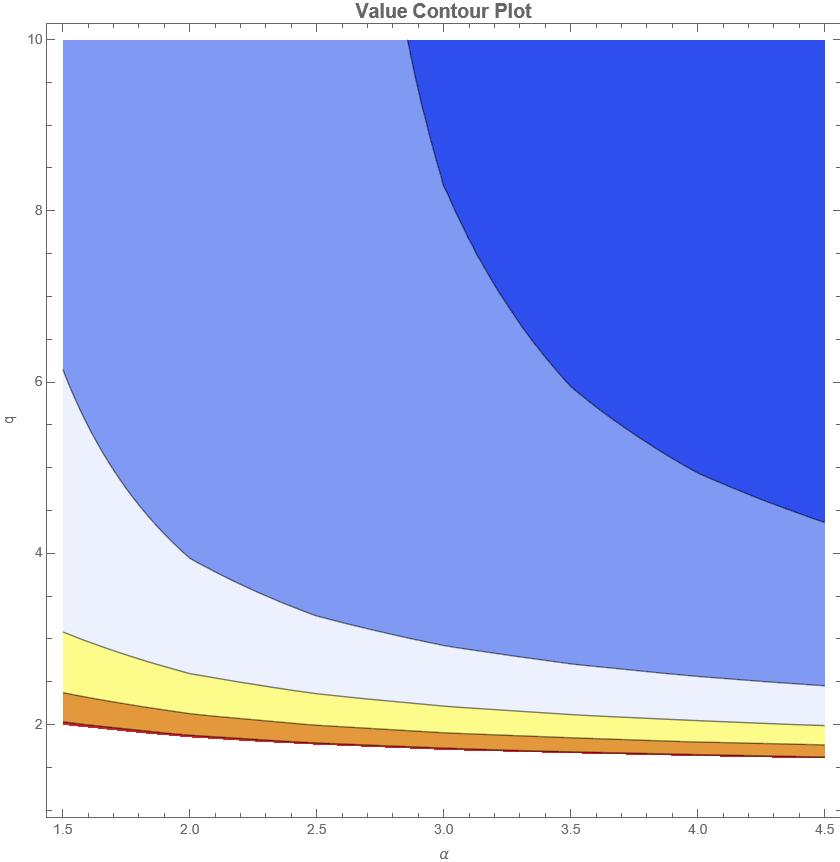}\label{rhs}
		\caption{RHS of Sharp Hausdorff Young inequality (Theorem \ref{sharp}) }
		
		\includegraphics[width=7 cm] {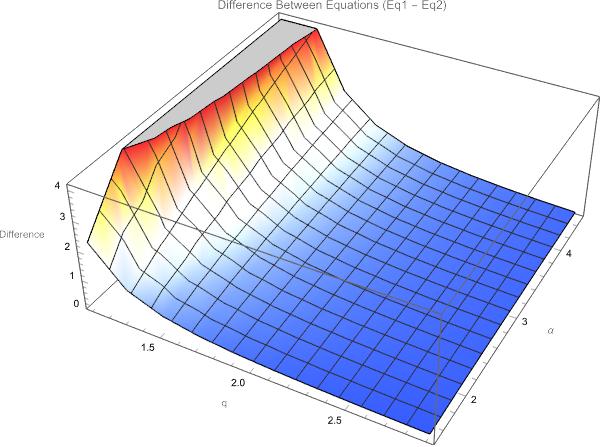}\label{diffshy}
		\caption{Diffrence of Sharp Hausdorff Young inequality (Theorem \ref{sharp})}
		
		\includegraphics[width=7 cm] {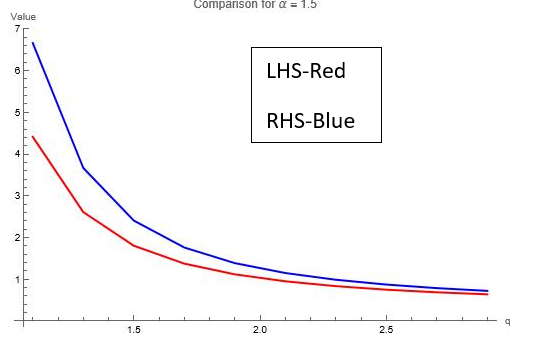}\label{diffg4} 
		\caption{Diffrence at $\alpha$= 1.5 of Sharp Hausdorff Young inequality (Theorem \ref{sharp})}
	\end{figure}
	\begin{table}[h]  
		\caption{Numerical simulation of Gaussian function for Theorem \ref{sharp} 
			\label{table-sharp} }  
		\begin{tabular}{|c|c|c|c|c|c|}  
			\hline                         
			\textbf{$\alpha$} & \textbf{q}& \textbf{RHS}& \textbf{LHS} &{Difference} \\  
			\hline                       
			1.5 & 1.1 & 6.66269& 4.41178 & 2.25091 \\  
			\hline                        
			& 1.3& 3.66655 &2.60754 & 1.05901\\  
			\hline 
			&1.5&2.40652&1.80171&0.603909\\
			\hline 
			& 1.7& 1.75878& 1.37245 &0.386327\\
			\hline
			& 1.9&1.38405 &1.11526 & 0.268789 \\
			\hline
		\end{tabular}\\
		-\\
		\begin{tabular}{|c|c|c|c|c|}  
			\hline                        
			\textbf{$\alpha$} & \textbf{q} & \textbf{RHS}& \textbf{LHS}& Difference \\  
			\hline  
			2& 1.1& 8.2577&4.29789&3.95981\\
			\hline                       
			& 1.3& 4.22614&2.44005 & 1.78609\\  
			\hline                        
			& 1.5&2.62587  & 1.63696&0.988005 \\  
			\hline  
			& 1.7 & 1.84185&1.21914&0.622718 \\  
			\hline 
			& 1.9& 1.40403&0.97316& 0.430849\\
			\hline 
		\end{tabular}\\
		-\\
		\begin{tabular}{|c|c|c|c|c|}  
			\hline                        
			\textbf{$\alpha$} & \textbf{q} & \textbf{RHS}& \textbf{LHS}& Difference \\  
			\hline  
			2.5& 1.1& 9.84861&4.21159&5.63702\\
			\hline                       
			& 1.3& 4.7368&2.31758 & 2.41922\\  
			\hline                        
			& 1.5&2.81674  & 1.51962&1.29712 \\  
			\hline  
			& 1.7 & 1.9143&1.11211&0.802186 \\  
			\hline 
			& 1.9& 1.41959&0.875571& 0.544021\\
			\hline 
		\end{tabular}
	\end{table}
	\begin{remark}
		From the Table \ref{table-sharp} and Fig. $4, 5, 6$ and $7$, we observe that for $\alpha=1.5$ difference of LHS and RHS is decreasing when of q is near 2 or more than 2 increasing. Another interesting observation is that when $q$ is near 2 and the difference between the LHS and RHS decreases, the value of $\alpha$ also decreases..
	\end{remark}
	\section{Quaternions}
	Here, we focus on real quaternions. The set of real quaternions, denoted by $\mathbb{H}$. Every element in  $\mathbb{H}$ can be represented in the following form \cite{bhri,he}.\\
	\begin{center}
		$\mathbb{H}$=$\{p_1 +\boldsymbol{i}p_2 +\boldsymbol{j}p_3 +\boldsymbol{k}p_4 \colon p_1,p_2,p_3,p_4\in \mathbb{R}\}$,
	\end{center} 
	where the imaginary units \textbf{i, j} and \textbf{k} satisfy the following  
	\begin{eqnarray*}
		\textbf{	ij = -ji = k, jk = -kj = i,  ki = -ik = j},
	\end{eqnarray*}
	\begin{eqnarray}
		\boldsymbol{ i^2=j^2=k^2=ijk=}-1.
	\end{eqnarray}
	For quaternion $p=p_1+\textbf{i}p_2+\textbf{j}p_3+\textbf{k}p_4\in \mathbb{H}$,$p_1$ is the scalar part of $p$ denoted by $Sc(p)$ and $\boldsymbol{p}=\boldsymbol{i}p_2+\boldsymbol{j}p_2+\boldsymbol{k}p_4$ is called the vector part of $p$ denoted by $Vec(p)$. Let $p,r\in\mathbb{H}$ and $\boldsymbol{p,r}$ be their vectors parts respectively. The multiplication of two quaternions $p.r$ is defined as 
	\begin{eqnarray}
		pr	=p_1r_1-\boldsymbol{p.r}+p_1\boldsymbol{r}+p_1\boldsymbol{p}+\boldsymbol{p\times r}
	\end{eqnarray} 
	where \begin{eqnarray*}
		\boldsymbol{p.r}= p_2r_2+p_3r_3+p_4r_4 
	\end{eqnarray*} and
	\begin{eqnarray*}
		\boldsymbol{p\times r} =i(p_3r_4-p_4r_3)+j(p_4r_2-p_2r_4)+k(p_2r_3-p_3r_2).
	\end{eqnarray*}
	The quaternionic conjugation $\overline{p}$ is given by 
	\begin{eqnarray}\label{conju}
		\overline{p}=p_1-\boldsymbol{i}p_2-\boldsymbol{j}p_3-\boldsymbol{k}p_4.
	\end{eqnarray}
	Using \eqref{conju} we have, \begin{eqnarray*}
		\overline{pr}=\bar{r}\bar{p}.
	\end{eqnarray*} 
	\subsection{Split quaternion and properties}\cite{he}
	For any two quaternion units \(v\) and \(\mu\) satisfying \(v^{2} = \mu^{2} = -1\), a quaternion \(q\) can be written as
	\begin{eqnarray*}
		q=q_-+q_+, q_{\pm}= \frac{1}{2}(q\pm vq\mu).
	\end{eqnarray*}
	If $v=\mu$, then any quaternion $q$ can be split up to the commuting and anti-commuting components with respect to $v$, that is
	\begin{eqnarray}
		vq_-=q_-v,\label{6.44}\\ vq_+=-q_+v.\label{6.55}
	\end{eqnarray}
	Using equations\eqref{6.44} and\eqref{6.55}, we have 
	\begin{eqnarray}
		q\pm e^{v\theta}= e^{\mp v\theta}q\pm.
	\end{eqnarray}
	Also for $v\neq \pm\mu$, we have 
	\begin{eqnarray*}
		e^{\alpha v }q\pm e^{\beta\nu}=e^{(\alpha\mp\beta)v}q\pm=q\pm e^{(\beta\mp\alpha)\mu}.
	\end{eqnarray*}
	Specifically, if we choose \(v = \mathbf{i}\) and \(\mu = \mathbf{j}\), the preceding expression simplifies to
	\begin{eqnarray*}
		q=q_-+q_+,q\pm=\frac{1}{2}(q\pm \boldsymbol{i}q\boldsymbol{j}),
	\end{eqnarray*}
	which is known as orthogonal plane split(OPS) of a quaternion.
	\begin{lemma}\label{7.1}For $q\in \mathbb{H}$, we have
		\begin{eqnarray*}\label{2.7}
			|q|^2=|q_+|^2+|q_-|^2.
		\end{eqnarray*}
		Moreover, considering two distinct quaternions, we obtain $p,r \in \mathbb{H}$
		\begin{eqnarray*}
			Sc(r_+\bar{p_-})=0.
		\end{eqnarray*}
	\end{lemma}
	\begin{definition}
		A quaternion-valued function 
		\begin{eqnarray*}
			f:\mathbb{R}^2\to\mathbb{H}
		\end{eqnarray*}
		can be expressed as \begin{eqnarray*}
			f(t_1,t_2)=f_1(t_1,t_2)+\boldsymbol{i}f_2(t_1,t_2)+\boldsymbol{j}f_3(t_1,t_2)+\boldsymbol{k}f_4(t_1,t_2),
		\end{eqnarray*}
		where $f_n:\mathbb{R}^2\to \mathbb{R}$ for $n=1,2,3,4$. 
	\end{definition}
	\begin{definition}
		For $p\in[1,\infty)$ and $p\to \infty$, the $L^p$ norm of $f$ is defined as 
		\begin{eqnarray*}
			||f||_p=\begin{cases}
				(\int_{\mathbb{R}^2}|f(t_1,t_2)|^pd(t_1,t_2))^\frac{1}{p} ,& 1\leq p< \infty\\ ess\sup|f(t)|,&  p=\infty
			\end{cases}
		\end{eqnarray*}
	\end{definition}
	\begin{definition}
		We can define the quaternion value inner product 
		\begin{eqnarray}\label{2.13}
			(f,g)=\int_{\mathbb{R}^2}f(t_1.t_2)\overline{g(t_1,t_2)}d(t_1.t_2)
		\end{eqnarray}
		with symmetric real scalar part 
		\begin{eqnarray}\label{2.14}
			\langle g,f\rangle= \frac{1}{2}[(f,g)+(g,f)]=\langle f,g\rangle=\int_{\mathbb{R}^2}\langle f(t_1,t_2)\overline{g(t_1,t_2)}\rangle_0d(t_1,t_2),
		\end{eqnarray}
		Here, \(\langle \cdot \rangle_0\) represents the scalar part obtained from the cyclic multiplication of quaternions.
		\begin{eqnarray*}
			Sc(qrs)=\langle qrs\rangle_0 =\langle rsq\rangle_0=Sc(rsq).
		\end{eqnarray*}
		From equations \eqref{2.13} and \eqref{2.14}, we get \begin{eqnarray*}
			|| f||_2=\sqrt{\langle f,f\rangle}=\bigg(\int_{\mathbb{R}^2}|f(t_1,t_2)|^2d(t_1,t_2)\bigg)^\frac{1}{2}.
		\end{eqnarray*}
	\end{definition}
	\begin{lemma}
		For $f\in L^2(\mathbb{R}^2, \mathbb{H})$, the QOLCT of f can be expressed by orthogonal plane split ($OPS$) method as follows 
		\begin{eqnarray}
			\mathcal{O}^\mathbb{H}_{M_1}M_2[f](u)&=& \mathcal{O}^\mathbb{H}_{M_1,M_2}[f_++f_-](u)\\&=&(\mathcal{O}_{M_1M_2}f_+)(u)+(\mathcal{O}_{M_1M_2}f_-)(u).
		\end{eqnarray}
	\end{lemma}
	The reduced formulation of QOLCT is 
	\begin{eqnarray}
		&&(\mathcal{O}_{M_1M_2}^\mathbb{H}f\pm)(u)\nonumber\\&=&\frac{1}{2\pi \sqrt{-b_1(\pm b_2)}}\nonumber\\&\times&\int_{\mathbb{R}^2}e^{\frac{i}{2b_1}(a_1 t_1^2 + 2 t_1 (\tau_1-u_1)-2u_1(d_1\tau_1-b_1\eta_1)+ d_1u_1^2+ d_1\tau_1^2)\mp\frac{i}{2b_2}(a_2 t_2^2 + 2 t_2 (\tau_2-u_2)-2u_2(d_2\tau_2-b_2\eta_2))+ d_2u_2^2+ d_2\tau_2^2)}f_\pm dt\nonumber\\
		&=&
		\frac{1}{2\pi \sqrt{-b_1(\pm b_2)}}\nonumber \\&\times&\int_{\mathbb{R}^2}e^{\frac{i}{2b_1}(a_1 t_1^2 + 2 t_1 (\tau_1-u_1)-2u_1(d_1\tau_1-b_1\eta_1)+ d_1u_1^2+ d_1\tau_1^2)+\frac{i}{2(\mp b_2)}(a_2 t_2^2 + 2 t_2 (\tau_2-u_2)-2u_2(d_2\tau_2-b_2\eta_2))+ d_2u_2^2+ d_2\tau_2^2)}f_\pm dt.\label{qolctpr}
	\end{eqnarray}
	This lemma is important in terms of proving the uncertainty principle and inequalities of QOLCT.
	\begin{remark}
		In \eqref{qolctpr}	$(\mathcal{O}_{M_1M_2}^\mathbb{H}f_+)(u)$ is 2D OLCT of function $f_+(t)$ w.r.t $M_1$ = $(a_1,b_1,c_1,d_1,\tau_1.\eta_1)$ and $M_2$ = $(a_2,-b_2,-c_2,d_2,\tau_2.\eta_2)$ and similar way $f_-(t)$ w.r.t. $M_1$ = $(a_1,b_1,c_1,d_1,\tau_1.\eta_1)$ and $M_2$ = $(a_2,b_2,c_2,d_2,\tau_2.\eta_2)$.
	\end{remark}
	\begin{lemma}
		Modulation identity.For $f:\mathbb{R}^2\to \mathbb{H}$ using Lemma~\ref{7.1}, we have the following identities:
		\begin{eqnarray}
			|f(t_1,t_2)|^2=|f_+(t_1,t_2)|^2+|f_-(t_1,t_2)|^2.
		\end{eqnarray}
		\begin{eqnarray} \label{ops}
			|\mathcal{O}_{M_1M_2}^\mathbb{H}f(u)|^2=	|\mathcal{O}_{M_1M_2}^\mathbb{H}f_+(u)|^2+	|\mathcal{O}_{M_1M_2}^\mathbb{H}f_-(u)|^2.
		\end{eqnarray}
	\end{lemma}
	\begin{proposition}
		Parseval theorem: The scalar product of two quaternion functions $f,g \in L^2(\mathbb{R}^2,\mathbb{H})$ \cite{hit} is given by scalar product of the corresponding QOLCT   $\mathcal{O}_{M_1M_2}^\mathbb{H}f$ and  $\mathcal{O}_{M_1M_2}^\mathbb{H}g$, as
		\begin{eqnarray}
			\langle\mathcal{O}_{M_1M_2}^\mathbb{H}f,\mathcal{O}_{M_1M_2}^\mathbb{H}g \rangle=\langle f,g\rangle .
		\end{eqnarray}
		\begin{corollary}
			Plancherel's identity: For $f\in L^2(\mathbb{R}^2,\mathbb{H})$, we have 
			\begin{eqnarray}
				||f||^2_{L^2(\mathbb{R}^2,\mathbb{H})}=||\mathcal{O}_{M_1M_2}^\mathbb{H}f||^2_{L^2(\mathbb{R}^2,\mathbb{H})}.
			\end{eqnarray}
		\end{corollary}
	\end{proposition}
	
	\subsection{Two-sided quaternion integral transforms}\ \\
	The quaternion Fourier Transform (QFT), introduced by Ell \cite{ft2}, useful for signal processing and color images, as it transforms real 2D signals into quaternion-valued frequency domain signals. Generally QFT exists in three distinct forms, the left-sided, right-sided and the two-sided QFT. Here we discuss only two sided QFT.
	\begin{definition}
		(Two-sided QFT)\cite{huk} The QFT of $f\in L^1( \mathbb{R}^2, \mathbb{H})$ is defined as 
		\begin{eqnarray}
			(\mathcal{F}f)(u)=\frac{1}{2 \pi}\int_{\mathbb{R}^2}e^{-it_1u_1}f(t)e^{-jt_2u_2}dt.
		\end{eqnarray}
		The inversion of the two sided QFT $\mathcal{F}^\mathbb{H}f:\mathbb{R}^2\to \mathbb{H} $ is given by \begin{eqnarray*}
			f(t)=\frac{1}{2\pi}\int_{\mathbb{R}^2}e^{it_1u_1}(\mathcal{F}^\mathbb{H}f)(u)e^{jt_1u_2}du.
		\end{eqnarray*}
	\end{definition}
	\begin{definition}
		Let  $f \in L^1(\mathbb{R}^2) $ then two sided QOLCT \cite {pln} of a function $f$ is defined as 
		\begin{eqnarray}
			O^H_{M_1 M_2}[f(t)](u)=		
			\int_{\mathbb{R}^2} \mathcal{K}^i_{M_1}(u_1,t_1) \mathcal{K}^j_{M_2}(u_2,t_2)f(t) dt,
		\end{eqnarray}
		where $ \mathcal{K}^i_{M_1}(u_1,t_1)$ and  $\mathcal{K}^j_{M_2}(u_2,t_2)$ are kernel function of QOLCT is given by 
		\begin{eqnarray}
			\mathcal{K}^i_{M_1}(u_1,t_1) = 
			\begin{cases}
				\sqrt{\frac{1}{2\pi i b_1}}~e^{\frac{i}{2b_1}(a_1 t^2_1 + 2 t_1 (\tau_1-u_1)-2u_1(d_1\tau_1-b_1\eta_1)+ d_1u_1^2+ d_1\tau_1^2)}, & \hspace{-4mm} ~~ b_1\neq0\\
				\sqrt{d_1}~e^{i\frac{c_1d_1}{2}(u_1-\tau_1)^2+i\eta_1 u_1}f[d_1(u_1-\tau_1)] & \hspace{-0.3cm}, ~~b_1=0
				
			\end{cases}
		\end{eqnarray} 
		and
		\begin{eqnarray}
			\mathcal{K}^j_{M_2}(u_2,t_2) = 
			\begin{cases}
				\sqrt{\frac{1}{2\pi j b_2}}~e^{\frac{j}{2b_2}(a_2 t^2_2 + 2 t_2 (\tau_2-u_2)-2u_2(d_2\tau_2-b_2\eta_2)+ d_2u_2^2+ d_2\tau_2^2)}, & \hspace{-4mm} ~~ b_2\neq0\\
				\sqrt{d_2}~e^{j\frac{c_2d_2}{2}(u_2-\tau_2)^2+j\eta_2 u_2}f[d_2(u_2-\tau_2)] & \hspace{-0.3cm}, ~~b_2=0
				
			\end{cases}
		\end{eqnarray}
		The inversion of the two sided QOLCT is defined as 
		\begin{eqnarray*}
			f(t)=	\int_{\mathbb{R}^2} \mathcal{K}^{-i}_{M_1}(u_1,t_1)\mathcal{O}_{M_1M_2}^\mathbb{H}[f(t)](u) \mathcal{K}^{-j}_{M_2}(u_2,t_2)du.
		\end{eqnarray*}
	\end{definition}
	\subsection{Connection between QOLCT and QFT}
	
	Let  $f \in L^1(\mathbb{R}^2,\mathbb{H}) $. Then
	\begin{eqnarray*}
		\mathcal{O}_{M_1 M_2}^H(f)(u_1,u_2)&=&\int_{\mathbb{R}^2} \mathcal{K}^i_{M_1}(u_1,t_1)f(t) \mathcal{K}^j_{M_2}(u_2,t_2) dt\\
		&=&\int_{\mathbb{R}^2}\sqrt{\frac{1}{2\pi i b_1}}~e^{\frac{i}{2b_1}(a_1 t^2_1 + 2 t_1 (\tau_1-u_1)-2u_1(d_1\tau_1-b_1\eta_1)+ d_1u_1^2+ d_1\tau_1^2)}\\&\times&f(t)\sqrt{\frac{1}{2\pi j b_2}}~e^{\frac{j}{2b_2}(a_2 t^2_2 + 2 t_2 (\tau_2-u_2)-2u_2(d_2\tau_2-b_2\eta_2)+ d_2u_2^2+ d_2\tau_2^2)}dt\\
		&=&\frac{1}{\sqrt{ib_1}}~e^{\frac{i}{2b_1}[d_1u_1^2+d_1\tau_1^2-2u_1(d_1\tau_1-b_1\eta_1)]}\mathcal{F}\{e^{\frac{i}{2b_1}(a_1t_1^2+2t_1\tau_1)}f(t)e^{\frac{j}{2b_2}(a_2t_2^2+2t_2\tau_2)}\}(\frac{u}{b})\\&\times&\frac{1}{\sqrt{jb_2}}~e^{\frac{i}{2b_2}[d_2u_2^2+d_2\tau_2^2-2u_2(d_2\tau_2-b_2\eta_2)]}.
	\end{eqnarray*}
	\section{Inequalities and Uncertainty principle on QOLCT}
	In this section, we investigate the uncertainty principles and  inequalities for the two-sided QOLCT using the OPS method, highlighting their significance.
	\subsection{Sharp Young- Hausdorff inequality for two sided QOLCT}
	\begin{theorem}
		Let p$\in[1,2]$ and q be such that $\frac{1}{p}+\frac{1}{q}=1$. Then 
		\begin{eqnarray*}
			||\mathcal{O}_{M_1M_2}^\mathbb{H}f||_{L^{q}(\mathbb{R}^2,\mathbb{H})}\leq|b_1 b_2|^{\frac{1}{q}-\frac{1}{2}} \left( \frac{p^{1/p}}{q^{1/q}} \right) (2\pi)^{\frac{1}{q} - \frac{1}{p}} ||f||_{L^{p}(\mathbb{R}^2,\mathbb{H})}.
		\end{eqnarray*}
	\end{theorem}
	\begin{proof}
		This theorem is considered in three categories, determined by the value of $p$.\\
		(Case i) When $p=1$ 
		\begin{eqnarray*}
			||\mathcal{O}_{M_1M_2}^\mathbb{H}f||_{L^{\infty}(\mathbb{R}^2,\mathbb{H})}&=& ess\ supp\big|\mathcal{O}_{M_1M_2}^\mathbb{H}f\big|\\&\leq&\big|\frac{1}{\sqrt{b_1b_2}}\big|||f||\\&=&\mathcal{K}||f||_{L^{1}(\mathbb{R}^2,\mathbb{H})}.
		\end{eqnarray*}
		(Case ii) When $ p=2$,  $f\in{L^{2}(\mathbb{R}^2,\mathbb{H})}$, using Parseval identity
		\begin{eqnarray*}
			||\mathcal{O}_{M_1M_2}^\mathbb{H}f(u)||_{L^{2}(\mathbb{R}^2,\mathbb{H})}=||f(t)||_{L^{2}(\mathbb{R}^2,\mathbb{H})}.
		\end{eqnarray*}
		(Case iii ) When $1<p<2$ and $1<\frac{q}{2}$, using modulus identity \eqref{ops}, we have 
		\begin{eqnarray*}
			||\mathcal{O}_{M_1M_2}^\mathbb{H}||_{L^{q}(\mathbb{R}^2,\mathbb{H})}&\leq& \left( \int_{\mathbb{R}^2}|\mathcal{O}_{M_1M_2}^\mathbb{H} (f)(u)|^{q}du\right)^{\frac{1}{q}}\\&=&\left( \int_{\mathbb{R}^2}\big[|\mathcal{O}_{M_1M_2}^\mathbb{H} (f)(u)|^{2}\big]^{\frac{q}{2}}du\right)^{\frac{1}{q}}\\&=&\left(\int_{\mathbb{R}^2}\big[|\mathcal{O}_{M_1M_2}^\mathbb{H}f_+|^2+|\mathcal{O}_{M_1M_2}^\mathbb{H}f_-|^2\big]^\frac{q}{2}du\right)^\frac{1}{q}.
		\end{eqnarray*}
		Using  Minkowski’s inequality and \eqref{qolctpr} and \eqref{sharp}, we have
		\begin{eqnarray*}
			||\mathcal{O}_{M_1M_2}^\mathbb{H}||_{L^{q}(\mathbb{R}^2,\mathbb{H})}	&\leq& \left[\left(\left(\int_{\mathbb{R}^2}|\mathcal{O}_{M_1M_2}^\mathbb{H}f_+|^2\right)^{q/2}du\right)^{2/q}+\left(\left(\int_{\mathbb{R}^2}|\mathcal{O}_{M_1M_2}^\mathbb{H}f_-|^2\right)^{q/2}du\right)^{2/q}\right]^\frac{1}{2}\\
			&=&\left[\left(\int_{\mathbb{R}^2}|\mathcal{O}_{M_1M_2}^\mathbb{H}f_+|^{q}du\right)^{2/q}+\left(\int_{\mathbb{R}^2}|\mathcal{O}_{M_1M_2}^\mathbb{H}f_-|^q\right)^{2/q}du\right]^\frac{1}{2}\\&\leq&\left[\left(\left(\int_{\mathbb{R}^2}|\mathcal{O}_{M_1M_2}^\mathbb{H}f_+|^q\right)^{1/q}du\right)^{2}+\left(\left(\int_{\mathbb{R}^2}|\mathcal{O}_{M_1M_2}^\mathbb{H}f_-|^q\right)^{1/q}du\right)^{2}\right]^\frac{1}{2}\\
			&=&\left[\left(||\mathcal{O}_{M_1M_2}^\mathbb{H}f_+||_{({L^q},\mathbb{H})}\right)^2+\left(||\mathcal{O}_{M_1M_2}^\mathbb{H}f_-||_{({L^p},\mathbb{H})}\right)^2\right]^\frac{1}{2}\\
			&=&\left[\left(\mathcal{K}||f_+||_{({L^p},\mathbb{H})}\right)^2+\left(\mathcal{K}||f_-||_{({L^p},\mathbb{H})}\right)^2\right]^\frac{1}{2}\\
			&=&\mathcal{K}||f||_{({L^p},\mathbb{H})}.
		\end{eqnarray*}
		Accordingly, the stated result is established.
	\end{proof}
	\subsection{Pitt's inequality for two sided QOLCT }
	\begin{theorem}
		Let $f\in\mathbb{S}(\mathbb{R}^2,\mathbb{H})$ and $0\leq \lambda<2 $ the following holds
		\begin{eqnarray*}
			\int_{\mathbb{R}^2}|u|^{-\lambda}|\mathcal{O}_{M_1M_2}^\mathbb{H}f(u)|^2du\leq\frac{C_\lambda}{|b_1b_2|^2}\int_{\mathbb{R}^2}|t|^\lambda|f(t)|^2dt. 
		\end{eqnarray*}
	\end{theorem}
	\begin{proof}
		Using \eqref{ops},\eqref{3.44} and \eqref{qolctpr}, we get
		\begin{eqnarray*}
			\int_{\mathbb{R}^2}|u|^{-\lambda}|\mathcal{O}_{M_1M_2}^\mathbb{H}f(u)|^2&=&\int_{\mathbb{R}^2}
			|u|^{-\lambda}|\mathcal{O}_{M_1M_2}^\mathbb{H}f_+(u)|^2du+\int_{\mathbb{R}^2}
			|u|^{-\lambda}|\mathcal{O}_{M_1M_2}^\mathbb{H}f_-(u)|^2du
			\\&\leq&\frac{C_\lambda}{|b_1b_2|^2}\int_{\mathbb{R}^2}|t|^\lambda|f_+(t)|^2dt+\frac{C_\lambda}{|b_1b_2|^2}\int_{\mathbb{R}^2}|t|^\lambda|f_-(t)|^2dt\\&=&\frac{C_\lambda}{|b_1b_2|^2}\int_{\mathbb{R}^2}|t|^\lambda|f(t)|^2dt.
		\end{eqnarray*}
		This completes the proof.
	\end{proof}
	In particular, when  $\lambda = 0 $, Pitt’s inequality simplifies, yielding the logarithmic uncertainty Principle as a limiting result.
	The concept of the logarithmic uncertainty principle was initially proposed by W. Beckner in 1995. This result is obtained by applying Pitt’s inequality.
	\begin{eqnarray}
		\int_{\mathbb{R}^2}\ln|u||\mathcal{O}_{M_1M_2}f(u)|^2du+\int_{\mathbb{R}^2}\ln |t||f(t)|^2dt\geq C'_0\int_{\mathbb{R}^2}|f(t)|^2dt.
	\end{eqnarray}
	
	\subsection{Logarithmic uncertainty principle for  two sided QOLCT} 
	\begin{theorem}
		
		Let $f\in L^2(\mathbb{R}^2,\mathbb{H})$, Then
		\begin{eqnarray}
			\int_{\mathbb{R}^2}\ln|u||\mathcal{O}_{M_1M_2}^\mathbb{H}f(u)|^2du+\int_{\mathbb{R}^2}\ln |t||f(t)|^2dt\geq K'_0\int_{\mathbb{R}^2}|f(t)|^2dt.
		\end{eqnarray}
		\begin{proof}
			As $f_+$ and $f_-$ are 2D orthogonal functions satisfying inequality \eqref{2dlogrithmic} \eqref{ops}, the result is obtained directly. 
			\begin{eqnarray}\label{7.31}
				\int_{\mathbb{R}^2}\ln|u||\mathcal{O}_{M_1M_2}^\mathbb{H}f_+(u)|^2du+\int_{\mathbb{R}^2}\ln |t||f_+(t)|^2dt\geq K'_0\int_{\mathbb{R}^2}|f_+(t)|^2dt,
			\end{eqnarray}
			and
			\begin{eqnarray}\label{7.32}
				\int_{\mathbb{R}^2}\ln|u||\mathcal{O}_{M_1M_2}^\mathbb{H}f_-(u)|^2du+\int_{\mathbb{R}^2}\ln |t||f_-(t)|^2dt\geq K'_0\int_{\mathbb{R}^2}|f_-(t)|^2dt.
			\end{eqnarray} 
			Combining \eqref{7.31} and \eqref{7.32}, we can rewrite as
			\begin{eqnarray*}
				\int_{\mathbb{R}^2}\ln|u||\mathcal{O}_{M_1M_2}^\mathbb{H}f(u)|^2du+\int_{\mathbb{R}^2}\ln |t||f(t)|^2dt&=&	\int_{\mathbb{R}^2}\ln|u||\mathcal{O}_{M_1M_2}^\mathbb{H}f_+(u)|^2du+\int_{\mathbb{R}^2}\ln |t||f_+(t)|^2dt\\&+&	\int_{\mathbb{R}^2}\ln|u||\mathcal{O}_{M_1M_2}^\mathbb{H}f_-(u)|^2du+\int_{\mathbb{R}^2}\ln |t||f_-(t)|^2dt\\&\geq&K_0'\int_{\mathbb{R}^2}|f_+(t)|^2dt+K_0'\int_{\mathbb{R}^2}|f_-(t)|^2dt\\&=& K_0'\int_{\mathbb{R}^2}|f(t)|^2dt.
			\end{eqnarray*}
		\end{proof}
	\end{theorem}
	Hence, this completes the proof.
	\subsection{Entropy uncertainty principle for   two sided QOLCT}
	\begin{theorem}
		Let $\mathcal{O}_{M_1M_2}^\mathbb{H}f(u)$ be the two sided OLCT of any $f\in\mathbb{S}(\mathbb{R}^2, \mathbb{H})$ such that $||f(t)||=1$, then \begin{eqnarray*}
			\mathcal{E}(|f(t)|^2)+|b_1b_2|\mathcal{E}(|\mathcal{O}_{M_1M_2}^\mathbb{H}f(u)|^2)&\geq&\ln(\pi e |b_1b_2|^{|b_1b_2|}).
		\end{eqnarray*}
	\end{theorem}
	\begin{proof}
		Let \( f \in \mathbb{S}(\mathbb{R}^2, \mathbb{H}) \) such that \( \|f\|_{2} = 1 \). Then, by the modulation identity, we obtain
		\begin{eqnarray*}
			1=||f||_2^2&=&\int_{\mathbb{R}^2}|f(t)|^2dt\\&=&\int_{\mathbb{R}^2}|f_+(t)|dt+\int_{\mathbb{R}^2}|f_-(t)|dt\\ i.e., 1&=&||f_+||^2_2+||f_+||^2_2
		\end{eqnarray*}
		Let $||f_+||= \gamma$ then  $||f_-||= 1-\gamma$.  Let  $h_1=\frac{f_+}{\gamma^\frac{1}{2}} $ and $h_2=\frac{f_-}{1-\gamma^\frac{1}{2}} $. \\ This implies   $|f_+|^2=|\gamma||h_1|^2, |f_+|^2=|1-\gamma||h_2|^2$.\\ Using modulation identity on $\mathcal{O}_{M_1M_2}^\mathbb{H}f(u)$, we get \begin{eqnarray*}
			|\mathcal{O}_{M_1M_2}^\mathbb{H}f(u)|^2=\gamma|\mathcal{O}_{M_1M_2}^\mathbb{H}h_1(u)|^2+1-\gamma|\mathcal{O}_{M_1M_2}^\mathbb{H}h_2(u)|^2.
		\end{eqnarray*} 
		Using the theorem \ref{2deup}, we have \begin{eqnarray*}
			\mathcal{E}(|h_1(t)|^2)+|b_1b_2|\mathcal{E}(|\mathcal{O}_{M_1M_2}h_1(u)|^2)&\geq& \ln(\pi e |b_1b_2|^{|b_1b_2|})
		\end{eqnarray*}
		and \begin{eqnarray*}
			\mathcal{E}(|h_2(t)|^2)+|b_1b_2|\mathcal{E}(|\mathcal{O}_{M_1M_2}h_2(u)|^2)&\geq& \ln(\pi e |b_1b_2|^{|b_1b_2|}).
		\end{eqnarray*} 
		For real valued function
		$h(v)=-v\ln v$, we have
		\begin{eqnarray*} 
			h'(v)=-(1+\ln v)\\ Thus, h''=-\frac{1}{v}.
		\end{eqnarray*}
		As its second derivative is negative for $v>0$, $h$ is concave function.
		\\For $v,s>0$ and $\gamma\in[0,1]$ , we have \begin{eqnarray*}
			h((1-\gamma)v +\gamma s)&\geq& (1-\gamma)h(v)+\gamma h(s)\\
			i.e. -((1-\gamma)v+\gamma s)\ln((1-\gamma)v+\gamma s)&\geq& -(1-\gamma)v\ln v -\gamma s\ln s.
		\end{eqnarray*}
		Let $v=|h_2(t)|^2$ and $s=|h_1(t)|^2$ . Then
		\begin{eqnarray*}
			-((1-\gamma)|h_1(t)|^2+\gamma |h_1|^2)\ln((1-\gamma)|h_2(t)|^2+\gamma |h_2|^2)&\geq& -(1-\gamma)|h_2(t)|\ln |h_2|^2 -\gamma |h_1|^2\ln |h_1|^2.
		\end{eqnarray*}
		Integrating both sides w.r.t. t, we obtain 
		\begin{eqnarray*}
			-\int_{\mathbb{R}^2}((1-\gamma)|h_2(t)|^2+\gamma)\ln((1-\gamma)|h_2(t)|^2+\gamma|h_1(t)|^2)dt\\ \geq -\int_{\mathbb{R}^2}(1-\gamma)|h_2(t)|^2\ln|h_2(t)|^2dt-\int_{\mathbb{R}^2}\gamma|h_1(t)|\ln|h_1(t)|^2dt.
		\end{eqnarray*}
		Using both side Shannon Entropy, we get
		\begin{eqnarray}\label{est}
			\mathcal{E}((1-\gamma)|h_2(t)|^2+\gamma|h_1(t)|^2)\geq (1-\gamma)\mathcal{E}|h_2(t)|^2+\gamma\mathcal{E}(|h_1(t )|^2)
		\end{eqnarray}
		Applying QOLCT in both sides, we obtain 
		\begin{eqnarray}
			\mathcal{E}((1-\gamma)|\mathcal{O}_{M_1M_2}^\mathbb{H}h_2(u)|^2+\gamma||\mathcal{O}_{M_1M_2}^\mathbb{H}h_1(u)|^2)\geq (1-\gamma)\mathcal{E}||\mathcal{O}_{M_1M_2}^\mathbb{H}h_2(u)|^2+\gamma\mathcal{E}(||\mathcal{O}_{M_1M_2}^\mathbb{H}h_1(u)|^2)\label{est2}.
		\end{eqnarray}
		Adding the equations \eqref{est}and \eqref{est2}, we get
		\begin{eqnarray*}
			&&\mathcal{E}((1-\gamma)|h_2(t)|^2+\gamma|h_1(t)|^2)+\mathcal{E}((1-\gamma)|\mathcal{O}_{M_1M_2}^\mathbb{H}h_1(u)|^2)\\ &\geq&(1-\gamma)\mathcal{E}(|h_2(t)^2|)+\gamma\mathcal{E}(|h_1|^2)+(1-\gamma)\mathcal{E}(|O_{M_1M_2}^\mathbb{H}h_2(t)|^2)+\gamma\mathcal{E}(|\mathcal{O}_{M_1M_2}^\mathbb{H}h_1(t)|^2)	\\&\geq&(1-\gamma)(\mathcal{E}|h_2(t)|^2+\mathcal{E}|\mathcal{O}_{M_1M_2}^\mathbb{H}h_2(t)|^2)+\gamma(\mathcal{E}|h_1(t)|^2+\mathcal{E}|\mathcal{O}_{M_1M_2}^\mathbb{H}h_1(t)|^2)	\\&\geq&(1-\gamma)\ln(\pi e|b_1b_2|^{|b_1b_2|}). 
		\end{eqnarray*}
		This completes the proof.
	\end{proof}
	
	\subsection	{Nazarov's Uncertainty principle for  two sided QOLCT}
	
	\begin{theorem} Let $\mathcal{O}_{M_1M_2}^\mathbb{H}f(u)$ be the two sided QOLCT of any $f\in L^2(\mathbb{R}^2,  \mathbb{H})$  and $T_1$ and $T_2$ be two finite measurable subsets of $\mathbb{R}^2$. Then
		\begin{eqnarray*}
			\int_{\mathbb{R}^2}|f(t)|^2dt\leq Ce^{c|T_1||T_2|}\left(\int_{\mathbb{R}^2\backslash T_1}|f(t)|^2dt+\int_{\mathbb{R}^2\backslash |b| T_2}|\mathcal{O}_{M_1M_2}^\mathbb{H}f(u)|^2du\right),
		\end{eqnarray*} 
		where $C>0$, $|T_1|$ and $|T_2|$ are Lebesgue measure of $T_1$ and $T_2$ respectively.
	\end{theorem}
	\begin{proof}
		Let $f_+$ and $f_-$ be two 2D orthogonal function. Then using \eqref{narzav2d} and \eqref{qolctpr}we have 
		\begin{eqnarray}\label{nup1}
			\int_{\mathbb{R}^2}|f_+|^2dt\leq C_1e^{C_1|T_1||T_2|}\left(\int_{\mathbb{R}^2\backslash T_1}|f_+(t)|^2dt+\int_{\mathbb{R}^2\backslash |b| T_2}|\mathcal{O}_{M_1M_2}^\mathbb{H}f_+(u)|du\right)
		\end{eqnarray} 
		and
		\begin{eqnarray}\label{nup2}
			\int_{\mathbb{R}^2}|f_-|^2dt\leq C_1e^{C_1|T_1||T_2|}\left(\int_{\mathbb{R}^2\backslash T_1}|f_-(t)|^2dt+\int_{\mathbb{R}^2\backslash |b| T_2}|\mathcal{O}_{M_1M_2}^\mathbb{H}f_-(u)|du\right).
		\end{eqnarray}
		Adding equations \eqref{nup1} and \eqref{nup2} we get
		\begin{eqnarray*}
			\int_{\mathbb{R}^2}|f_+|^2dt+\int_{\mathbb{R}^2}|f_-|^2dt\leq Ce^{C|T_1||T_2|}\left(\int_{\mathbb{R}^2\backslash T_1}|f_+|^2dt+\int_{\mathbb{R}^2\backslash T_1}|f_-|^2dt\right)+\\Ce^{C|T_1||T_2|}\left(\int_{\mathbb{R}^2\backslash |b| T_2}|\mathcal{O}_{M_1M_2}^\mathbb{H}f_+(u)|^2du+\int_{\mathbb{R}^2\backslash |b| T_2}|\mathcal{O}_{M_1M_2}^\mathbb{H}f_-(u)|^2du\right).
		\end{eqnarray*}
		Where $C=max\{C_1 ,C_2\}$. Then
		\begin{eqnarray*}
			\int_{\mathbb{R}^2}|f(t)|^2\leq Ce^{C|T_1||T_2|}\left(\int_{\mathbb{R}^2\backslash T_1}|f(t)|^2dt+\int_{\mathbb{R}^2\backslash |b| T_2}|\mathcal{O}_{M_1M_2}^\mathbb{H}f(u)|^2du\right).
		\end{eqnarray*}
		This completes the proof.
	\end{proof}
	\subsection{Heisenberg-type uncertainty principle for  two sided QOLCT}
	
	\begin{theorem}
		Let $f\in L^2(\mathbb{R}^2,\mathbb{H})$ be quaternion valued function such that $	D_{t_1^mt_2^n}\in L^2( \mathbb{R}^2,\mathbb{H})$.   Then 
		\begin{eqnarray*}
			\Pi(\Delta_{t_k}^\mathbb{H} \Delta_{u_k}^\mathbb{H})^2 \geq \Pi |\frac{b_k}{2}|^2||f||^4_2,
		\end{eqnarray*}
		where \begin{eqnarray*}
			\Delta_{t_k}^\mathbb{H}=||t_kf(t)||^2_2\\
			\Delta_{u_k}^\mathbb{H}=||u_k\mathcal{O}_{M_1M_2}^\mathbb{H}|| .
		\end{eqnarray*}
		For k=1,2
	\end{theorem}
	\begin{proof} For  $f\in L^2(\mathbb{R}^2,\mathbb{H})$, we have
		\begin{eqnarray*}
			|(	\Delta_{t_1}^\mathbb{H} \Delta_{u_1}^\mathbb{H})|^2&=&\int_{\mathbb{R}^2}|t_1|^2|f(t)|^2dt\int_{\mathbb{R}^2}|u_1|^2|\mathcal{O}_{M_1M_2}^\mathbb{H}f(u)|^2du\\&\geq& \left(\int_{\mathbb{R}^2}|t_1|^2|f_+(t)|^2dt+\int_{\mathbb{R}^2}|t_1|^2|f_-(t)|^2dt\right)\\&\times&\left(\int_{\mathbb{R}^2}|u_1|^2|\mathcal{O}_{M_1M_2}^\mathbb{H}f_+(u)|^2du+\int_{\mathbb{R}^2}|u_1|^2|\mathcal{O}_{M_1M_2}^\mathbb{H}f_-(u)|^2du\right)
		\end{eqnarray*}
		Consider four functions A, B, C, D as
		
		\begin{eqnarray}\label{htup}
			&&|(	\Delta_{t_1}^\mathbb{H} \Delta_{u_1}^\mathbb{H})|^2	\geq A+B+C+D,
		\end{eqnarray}
		where \begin{eqnarray*}
			&&A=\int_{\mathbb{R}^2}|t_1|^2|f_+(t)|^2dt\int_{\mathbb{R}^2}|u_1|^2|\mathcal{O}_{M_1M_2}^\mathbb{H}f_+(u)|^2du,
			\\&&B=\int_{\mathbb{R}^2}|t_1|^2|f_+(t)|^2dt\int_{\mathbb{R}^2}|u_1|^2|\mathcal{O}_{M_1M_2}^\mathbb{H}f_-(u)|^2du,
			\\&&C=\int_{\mathbb{R}^2}|t_1|^2|f_-(t)|^2dt\int_{\mathbb{R}^2}|u_1|^2|\mathcal{O}_{M_1M_2}^\mathbb{H}f_+(u)|^2du,
			\\&&D=\int_{\mathbb{R}^2}|t_1|^2|f_-(t)|^2dt\int_{\mathbb{R}^2}|u_1|^2|\mathcal{O}_{M_1M_2}^\mathbb{H}f_-(u)|^2du.	\end{eqnarray*}
		Using proposition \ref{prop3.6}, we see
		\begin{eqnarray*}
			A\geq \frac{b_1}{4}||f_+||^2_2\\ D\geq \frac{b_1}{4}||f_-||^2_2		
		\end{eqnarray*}
		Putting value of A  and D in equation \eqref{htup}, we get
		\begin{eqnarray*}
			|(	\Delta_{t_1}^\mathbb{H} \Delta_{u_1}^\mathbb{H})|^2	&\geq& \frac{b_1}{4}||f_+||^2_2+\frac{b_1}{4}||f_-||^2_2\\&\geq&\frac{b_1}{4}||f||^2_2.
		\end{eqnarray*}
		Similarly for $k=2$ 
		\begin{eqnarray*}
			|(	\Delta_{t_2}^\mathbb{H} \Delta_{u_2}^\mathbb{H})|^2\geq\frac{b_1}{4}||f||^2_2.
		\end{eqnarray*}
		This completes the proof.
	\end{proof}
	\section{Conclusion}
	This paper investigated inequalities and uncertainty principles for the OLCT. Shifting was found to have no impact on the sharp Young- Hausdorff inequality, the Heisenberg uncertainty principle, or the entropy uncertainty principle, but it affected Pitt’s inequality. Scaling influenced only the entropy uncertainty principle. We also presented numerical simulations that illustrated the Heisenberg-type uncertainty principle and the sharp Young-Hausdorff inequality. Furthermore, we demonstrated that the OPS method was the most effective and reliable approach for establishing inequalities and uncertainty principles for the QOLCT within the framework of the 2D-OLCT. Along with these current computation and simulation tools of OLCT, in our next work, we will do the simulation and explore the scaling and shifting effects in quaternion space also.
	\section*{Acknowledgment}
	The research facilities provided by SRM University-AP, Amaravati are acknowledge with thanks from the authors.
	\section*{Funding statement} 
	The author(s) received no financial support for the research, authorship, and/or publication of this article
	\section*{Declaration of competing interest}
	The authors declare that there are no conflicts of interest regarding the publication of this paper.
	\\\\
	\bibliographystyle{amsplain}				
	
\end{document}